\documentclass[12pt,A4,oneside,english]{amsart}
	\usepackage[hmargin=2cm,bottom=3cm,textwidth=17.3cm]{geometry}
\usepackage{amsthm}
\usepackage[hidelinks]{hyperref}
\usepackage{bbm}
\usepackage{cancel}
\newtheorem*{theorem*}{Theorem}
\newtheorem*{lemma*}{Lemma}
\newtheorem*{definition*}{Definition}
\numberwithin{equation}{section}

\usepackage{wrapfig}
\usepackage{amsmath,amsxtra}
\usepackage{amssymb}
\usepackage{xy}
\xyoption{all}
\usepackage{epsfig} \usepackage{epstopdf} %\usepackage[update]{epstopdf}
\usepackage{paralist}

\newcommand{\bra}{\langle}
\newcommand{\ket}{\rangle}
\newcommand{\quot}[1]{``#1''}

\newtheorem{fact}{Fact}[section]

\newtheorem{theorem}[fact]{Theorem}
\newtheorem{defi}[fact]{Definition}
\newtheorem{exer}[fact]{Exercise}
\newtheorem{exa}[fact]{Example}
\newtheorem{exas}[fact]{Examples}
\newtheorem{ob}[fact]{Observation}
\newtheorem{rremark}[fact]{Remark}
\newtheorem{proposition}[fact]{Proposition}
\newtheorem{corollary}[fact]{Corollary}
\newenvironment{remark}{\begin{rremark}\small \rm}{\end{rremark}}

\newenvironment{definition}{\begin{defi} \rm}{\end{defi}}
\newenvironment{example}{\begin{exa} \rm}{\renewcommand\qedsymbol{$\diamond$}\hfill\qedsymbol\end{exa}}

\newtheorem{conjecture}[fact]{Conjecture}

\newenvironment{introtheorem}[1]{\hphantom{}
	\\[.5\baselineskip]
	\noindent
	\textbf{Theorem #1.}\em
}{\smallskip}%\\[0.25\baselineskip]}

\newcommand{\R}{{\mathbb R}}
\newcommand{\C}{{\mathbb C}}

\newcommand{\Z}{{\mathbb Z}}

\newcommand{\F}{{\mathbb F}}

\renewcommand{\P}{{\mathbb P}}

\newcommand{\al}{\alpha}
\newcommand{\ka}{\kappa}
\newcommand{\si}{\sigma}
\newcommand{\se}{\subseteq}
\newcommand{\sub}{\backslash}

\DeclareMathOperator{\Hom}{Hom}

\DeclareMathOperator{\Ker}{Ker}

\DeclareMathOperator{\codim}{codim}

\DeclareMathOperator{\GL}{GL}

\DeclareMathOperator{\Sq}{Sq}

\DeclareMathOperator{\Gr}{Gr}

\DeclareMathOperator{\gr}{Gr}

\DeclareMathOperator{\tp}{tp}

\DeclareMathOperator{\cl}{\mathcal{C}\!\ell}

\newcommand{\csm}{{\rm c}^{\rm SM}}
\newcommand{\su}{\operatorname{w}}
\newcommand{\ssm}{{\rm s}^{\rm SM}}
\newcommand{\ssu}{{\rm s}^{\rm sw}}

\renewcommand{\epsilon}{\varepsilon}

	\newcommand*\clos[1]{\overline{#1}}
	\newcommand{\PP}{\mathbb{P}}
	\newcommand{\RP}{\mathbb{R}\mathbb{P}}
	
	\newcommand{\ga}{\gamma}

	\newcommand{\la}{\lambda}
	
	\newcommand{\Si}{\Sigma}

	\newcommand{\A}{\mathcal{A}}
	\newcommand{\stb}{,\ldots,}

	\newcommand{\CP}{\C \PP}
	
	\def\iso{\cong}
	
	\def\qed{{\hfill{$\Box$}}}

	\usepackage{tikz}

\renewcommand{\cl}{\operatorname{cl}}
\newcommand{\www}{w_1,w_2,\ldots}
\newcommand{\ttt}{t_1,t_2,\ldots}
\newcommand{\tatata}{\tau_1,\tau_2,\ldots}
\newcommand{\ccc}{c_1,c_2,\ldots}
\newcommand{\ssmeta}{\ssm_\eta}
\newcommand{\bw}{\bar{w}}

\makeatletter
\renewcommand*\l@subsection{\@tocline{2}{0pt}{2.5pc}{5pc}{}}
\makeatother

\address{\'Akos K.\ Matszangosz, HUN-REN Alfr\'ed R\'enyi Institute of Mathematics, Re\'altanoda utca 13-15, 1053 Budapest, Hungary}
\email{matszangosz.akos@gmail.com}

\address{L\'aszl\'o M. Feh\'er, E\"otv\"os Lor\'and University, Budapest, Hungary}
\email{lfeher63@gmail.com}

\thanks{\'A. K. M. is supported by the Hungarian National Research, Development and Innovation Office, NKFIH PD 145995 and NKFIH K 138828.}

\title[Obstructions for Morin and fold maps]{Obstructions for Morin and fold maps: Stiefel-Whitney classes and Euler characteristics of singularity loci}
\author{L\'aszl\'o M. Feh\'er, \'Akos K. Matszangosz}
\begin{document}
\begin{abstract} 
	For a singularity type $\eta$, let the \emph{$\eta$-avoiding number} of an $n$-dimensional manifold $M$ be the lowest $k$ for which there is a map $M\to\R^{n+k}$ without $\eta$ type singular points. For instance, the case of $\eta=\Si^1$ is the case of immersions, which has been extensively studied in the case of real projective spaces. In this paper we study the $\eta$-avoiding number for other singularity types. Our results come in two levels: first we give an abstract reasoning that a non-zero cohomology class is supported on the singularity locus $\eta(f)$, proving that $\eta(f)$ cannot be empty. Second, we interpret this obstruction as a non-zero invariant of the singularity locus $\eta(f)$ for \emph{generic} $f$. The main technique that we employ is Sullivan's Stiefel-Whitney classes, which are mod 2, real analogues of the Chern-Schwartz-MacPherson (CSM) classes. We introduce the \emph{Segre-Stiefel-Whitney classes of a singularity} $\ssu_\eta$ whose lowest degree term is the mod 2 Thom polynomial of $\eta$. Using these techniques we compute some universal formulas for the Euler characteristic of a singularity locus.
%	We compute universal Stiefel-Whitney classes of singularity loci, and derive obstructions to the existence of maps with only mild singularities, such as immersions, fold maps and Morin maps. We apply our results to compute the mod 2 Euler characteristics of such singularity loci. The key to these computations is a Borel-Haefliger type theorem which connects Stiefel-Whitney classes to Chern-Schwartz-MacPherson classes.
\end{abstract}
\maketitle
\setcounter{tocdepth}{2}

\tableofcontents
\section{Introduction}

The following question has been studied extensively \cite{Davis}: Let $M$ be an $n$-dimensional compact smooth manifold. For what $k$ does there exist an immersion of $M$ into $\R^{n+k}$? For future reference we will call the smallest such $k$ the \emph{immersion number} of $M$. For real and complex projective spaces there are several deep results, giving lower and upper bounds for the immersion number.

There is a natural generalization of this question: Let $\eta$ be a contact singularity. What is the smallest $k$ such that there is a map of $M$ into $\R^{n+k}$ with no $\eta$-points? We call this number the \emph{$\eta$-avoiding number} of $M$. For example for $\eta=\Sigma^1$, the $\eta$-avoiding number is the immersion number. 

In this paper we will study in detail two other examples: the case of \emph{cusps} $\eta=A_2$ and $\eta=\Sigma^2$. These are singularities with prototypes:
\[
A_2:(x,y)\mapsto (x^3+xy,y),\qquad \Si^2:(x,y)\mapsto (x^2,y^2).
\]
For $\eta=A_2$, maps with no $\eta$-points are called \emph{fold maps}, and for $\eta=\Sigma^2$ the maps with no $\eta$-points are called \emph{Morin maps}. We will call the corresponding $\eta$-avoiding number the \emph{fold number} of $M$, and the \emph{Morin number} of $M$, respectively.

We will give lower bounds for the fold and Morin numbers of real projective spaces using cohomological considerations. These arguments give trivial results for the immersion number, but even then they provide information on the geometry of the singularity locus. 
Using Stiefel-Whitney classes of singular spaces we will also give geometric interpretations of these lower bounds.

Our results come in two levels: First we give an abstract reasoning that a non-zero cohomology class is supported on the singularity locus $\eta(f)$, proving that $\eta(f)$ cannot be empty. Second, we interpret this obstruction as a non-zero invariant (e.g.\ odd Euler characteristic)  of the singularity locus $\eta(f)$ for \emph{generic} $f$. \\

The first level of our results is based on the notion of the \emph{stable avoiding ideal} $\mathcal{A}_\eta$ in $\F_2[w_1,w_2,\ldots]$, which is the collection of characteristic classes which are $\eta$-obstructions.
For $\eta=\Si^2$ and $\eta=A_2$, we find elements in this ideal which when evaluated on $\RP^n$ do not vanish, implying that there cannot exist an $\eta$-avoiding map in a given codimension.
 
The second level of our results is based on the Stiefel-Whitney class associated to a real algebraic variety (and more generally to certain stratified submanifolds), which is an extension of the mod 2 fundamental class. It contains higher degree terms, which can be interpreted --analogously to the Ohmoto-Aluffi theorem in the complex case---as the Euler characteristic of generic linear slices for subvarieties of $\RP^n$ (Theorem \ref{thm:real-aluffi}).

We define and compute \emph{Segre-Stiefel-Whitney} (Segre-SW) \emph{classes of contact singularities}, which is the real analogue of Ohmoto's Segre-Schwartz-MacPherson (Segre-SM) Thom polynomials. The lowest degree term of the Segre-SW class is the (mod 2) Thom polynomial of the singularity, which is classical and computes the fundamental class of the locus. Since both the Thom polynomial and the Segre-SW classes are elements in the avoiding ideal, these provide obstructions for the existence of $\eta$-maps.

Segre-SW classes of singularities can be explicitly computed using a Borel-Haefliger type theorem, which connects Segre-SM and Segre-SW classes, see Theorem \ref{thm:ssw}, \eqref{item:bh}. For instance, by results of Parusi\'nski-Pragacz \cite{ParusinskiPragacz} and Rim\'anyi \cite{tpp} on Segre-SM classes, we get explicit formulas for the Segre-SW classes of the $\Si^i$ and contact singularity types.

In particular, by computing such Segre-SW classes of singularity loci for maps $\RP^n\to \R^N$, we give formulas which express the Euler characteristic of the singularity locus. A typical result is the following:
\begin{introtheorem}{\ref{thm:euler-of-sigma2}}
		Assume that the binary expansion of $n$ contains no consecutive 1's. Let $p$ be the largest number, such that $2^p$ divides $n$, and assume that $p\geq 1$. Then for a generic smooth map $f:\RP^n\to \R^{n+l}$, where $l=n/2-2^{p-1}-1$ the locus $\Si^2(f)$ is a smooth manifold of dimension $d=2^{p}-2$ of odd Euler characteristic. In particular it is nonempty, and for $p\geq 2$ unorientable.
	\end{introtheorem}		

Our results suggest that there is a strong connection between Morin and fold maps. We show that the avoiding ideals of $\Si^2(l)$ and $A_2(l+1)$ are equal (Corollary \ref{cor:a2-si2}). In fact, we observe that their Segre-SW classes are equal in the range where we can compute both. Based on these observations we formulate a conjecture on a relationship between these two singularities, see Conjecture \ref{conj:fold_Morin}.

Finally, using these computations, it is possible to give universal formulas for the Euler characteristic of a singularity locus $\eta$, which we call its \emph{characteristic series} $\chi_\eta$. In \cite{SaekiSakuma1999}, Saeki and Sakuma show that the mod 2 Euler characteristic of the cusp points of a Morin map $M^4\to \R^4$ is equal to $\chi(M)$, the mod 2 Euler characteristic of $M$:
\begin{equation}\label{eq:chiA2_intro}
\chi(\clos{A_2}(M^4\to \R^4))=\tau_4[M]=\int_M w_4(M).
\end{equation}
 Using Segre-SW classes of singularity loci, we obtain similar results for generic maps. For instance, a typical result is the following (see Theorem \ref{thm:closure-cusp}):
	\[	\chi_{\clos{A_2}(0)}=\tau_2+0+\tau_4+0+(\tau_{4}\tau_2+\tau_3\tau_2\tau_1+\tau_6)+0+(\tau_4\tau_2^2 + \tau_6\tau_2 + \tau_4^2 + \tau_8)+\ldots\]
	Here the term $\tau_4$ corresponds to \eqref{eq:chiA2_intro}. For the precise formalism, see Section \ref{sec:eulerchar}.

These methods are useful for finding obstructions for $\eta$-maps between general source and target manifolds, we give some examples in Section \ref{sec:singularity}.

\subsection{Results and structure of the paper}

In Sections \ref{sec:SegreSullivan} to \ref{sec:avoiding} we establish the tools that we use to give bounds on the fold and Morin numbers of real projective spaces in the second half of the paper.

In Section \ref{sec:SegreSullivan}, we review the main properties of Stiefel-Whitney and Segre-Stiefel-Whitney classes of real algebraic varieties, and give a topological  interpretation of the coefficients of Stiefel-Whitney classes of a real subvariety of $\RP^n$ (Theorem \ref{thm:real-aluffi}). 

In Section \ref{sec:degloci} we define our main tool, the Segre-SW class of a real contact singularity and prove its main properties. We also give an effective method to calculate them. 

In Section \ref{sec:avoiding} we discuss a family of cohomological obstructions associated to a contact singularity, called its stable avoiding ideal.

In Section \ref{sec:obstructions_fold_morin} we  calculate lower bounds to the fold and Morin numbers of $\RP^n$ using their Thom polynomials and using  the corresponding avoiding ideals. In Theorem \ref{k=t+1} we show that the lower bound coming from the avoiding ideal is better (by 1) if the binary expansion of n contains no consecutive 1’s.

In Section \ref{sec:Sigmai} we interpret the lower bounds of Section \ref{sec:obstructions_fold_morin} as the Euler characteristic of the $\Si^2$-locus, see Theorem \ref{thm:euler-of-sigma2}. 

In Section \ref{sec:fold_vs_morin} we study the Euler characteristic of $A_2$-locus and formulate a conjectural relationship between fold and Morin maps. 

In Section \ref{sec:singularity} we compute Segre-SW classes of contact singularities $\eta$ other than $A_2$, and give universal formulas for the Euler characteristic of the $\eta$-locus of generic maps to Euclidean space. 

Finally, in Section \ref{sec:hierarchy} we study the relationship between the different obstructions considered in this paper and conclude with some further examples.

\textbf{Acknowledgements.} We are grateful to Rich\'ard Rim\'anyi for several discussions and for sharing with us his results \cite{tpp}, as well as informing us about the results of Brandyn Lee. We would like to thank J\"org Sch\"urmann for carefully explaining to us the proof of the transversal pullback property of Theorem \ref{thm:generic-pullback}. We thank Tam\'as Terpai for explaining his results to us, as well as for suggesting Example \ref{ex:DoldWu}. We thank Toru Ohmoto and Andrzej Weber for several discussions about the topics of this paper.

\section{Stiefel-Whitney classes of real algebraic varieties}\label{sec:SegreSullivan}
\subsection{Stiefel-Whitney classes and Euler characteristics}
Sullivan showed in \cite{sullivan} that the Stiefel-Whitney classes of a smooth manifold can be extended to real algebraic varieties. In particular, given a real algebraic subvariety $Z$ of the algebraic manifold $X$, there are classes $\su(Z)\in H^*(X;\F_2)$ with the property that
\begin{equation}\label{eq:su}
	\su(Z\se X)=[Z]+\ldots +\chi_2(Z)\cdot [pt].
\end{equation}
The main property that we will use in this paper is that the top degree term is the modulo 2 Euler characteristic $\chi_2(Z)$. 

Stiefel-Whitney classes can be also regarded as transformations from constructible functions to cohomology, which satisfy the Deligne-Grothendieck axioms, see \cite{fu-mccrory}.

\begin{remark}
	Sullivan's Stiefel-Whitney classes were defined as homology classes of real algebraic varieties, dual to the Stiefel-Whitney classes in the smooth case. Whitney \cite{Whitney1940} in his thesis defined these homology classes combinatorially for smooth manifolds, which were later rediscovered and developed by Cheeger \cite{Cheeger1969}. Sullivan and Akin \cite{sullivan}, \cite{Akin} gave a more general setting where these classes exist: a space $X$ is called \emph{Euler} if $\chi_2(X,X\sub x)=1$ for all $x\in X$. In particular, real algebraic varieties are Euler spaces as Sullivan showed \cite{sullivan}.
	
	Later, Fu and McCrory \cite{fu-mccrory} showed that these Stiefel-Whitney classes define transformations from subanalytic constructible functions to homology, and that for complex analytic varieties they are the mod 2 reductions of the Chern-Schwartz-MacPherson classes.
	
	Stiefel-Whitney homology classes are classically defined as homology classes, we will work in cohomology: via Borel-Moore homology this corresponds to the transformation $\cdot \cap [X]$.
	
	The most general spaces which admit Stiefel-Whitney classes 
	used in this paper are Whitney stratified, Euler spaces -- see \cite{sullivan} for the PL case and \cite{Trotman2020} for a survey of triangulability of Whitney stratifications. For further sources, see also \cite{McCrory}, \cite{fu-mccrory}, \cite{Parusinski}, \cite{KashiwaraSchapira}, \cite{Schurmann03}, \cite{bsy}.
\end{remark}
\begin{remark}\label{rmk:notation}
	We make a few remarks about terminology and notation. The classical terminology is \emph{Stiefel-Whitney homology classes} for the element in $H_*(Z)$ of a real algebraic variety $Z$, \cite{McCrory}, \cite{Parusinski}. We will call \emph{Stiefel-Whitney class} the corresponding cohomology class $\su(Z)\in H^*(X)$ when $Z$ is a subvariety (or, more generally, a stratified submanifold, which is also an Euler space)  of a smooth variety $X$. In notation, the Stiefel-Whitney class is $\su(Z\se X)\in H^*(X)$, and we will use the notation $w(Z)$ for the Stiefel-Whitney class for the total Stiefel-Whitney class of the tangent bundle $w(TZ)$. Notice that $\su(Z\se Z)=w(Z)$ whenever $Z$ is smooth.
\end{remark}
\begin{remark}
	We collect a few elementary facts about the modulo 2 Euler characteristic \[\chi_2(Z)=\sum_{i} \dim_{\F_2}H^i(Z;\F_2)\in \F_2\] that we will use without further mention. First, $\chi_2$ is the modulo 2 reduction of the Euler characteristic $\chi$, so the value of $\chi_2$ determines the parity of $\chi$. For an orientable compact manifold $M$ of dimension $4k+2$, $\chi_2(M)=0$ -- the parity of the Euler characteristic (by Poincar\'e duality) is the rank of $H^{2k+1}(M;\Z)$, which is even since the intersection form is symplectic.
\end{remark}

\subsection{Calculating Stiefel-Whitney classes: relation with CSM classes}
There is a wide variety of methods to calculate CSM classes. A Borel-Haefliger type theorem of Brasselet, Sch\"urmann and Yokura allows us to use these results to calculate Stiefel-Whitney classes in the algebraic case.

\begin{theorem}\cite[(0.10)]{bsy} \label{thm:bsy}
	Let $X$ be a smooth complexified variety, and let $\cl_\R$ denote Borel and Haefliger's real cycle class map \cite[5.12]{BorelHaefliger1961} from $A_\R(X)$, the Chow ring of real cycles of $X$ to the modulo 2 cohomology of the real points $X(\R)$:
	\begin{equation}\label{eq:clR}
		\cl_\R:A_\R(X)\to H^*(X(\R);\F_2),
	\end{equation}
	which maps a complexified cycle to the fundamental class of its real part. Then if $Z\se X$ is a complexified subvariety of $X$, then
	\begin{equation}\label{eq:BHw}
		\cl_\R(\csm(Z\se X))=\su(Z(\R)\se X(\R)).
	\end{equation}	
\end{theorem}
\begin{remark} Chern-Schwartz-MacPherson and Stiefel-Whitney classes are \emph{motivic invariants}, therefore they extend to \emph{constructible sets} (where constructible is meant in the appropriate category - complex algebraic \cite{bsy}, subanalytic \cite{fu-mccrory}, etc.). On constructible set we simply mean the difference $Z\setminus Y$, where $Y$ is a subvariety of $Z$.  Then $\csm((Z\setminus Y)\se X)=\csm(Z\se X)-\csm(Y\se X)$ for the complex case and $\su((Z\setminus Y)\se X)=\su(Z\se X)-\su(Y\se X)$.
	
	%\textbf{az vilagos, hogy ezek is Euler terek?}
\end{remark}

\subsection{Segre-Stiefel-Whitney classes: the transversal pullback property}

A strongly related notion is \emph{the Segre-Stiefel-Whitney class $\ssu(Z)\in H^*(X)$} - this is related to the Stiefel-Whitney class via the relation
\begin{equation}\label{eq:ssu_su}
	\su(Z\se X)=\ssu(Z\se X)\cdot w(X).
\end{equation}

Segre-Stiefel-Whitney classes satisfy the \emph{transversal pullback property}, see \cite{Matsui1988} and \cite{Schurmann2017} for the PL and the algebraic category respectively.
This property will be crucial in our geometric applications. We will be working in the smooth category, namely we are mainly interested in obstructions for the existence of smooth maps $f:\RP^n\to \R^m$ avoiding given singularities.

\begin{theorem}\label{thm:generic-pullback}
	Let $N$ be a real smooth algebraic variety and $Z\subset N$ be a subvariety. Let $f:M\to N$ be a generic smooth map from a smooth manifold $M$. Then $f^{-1}(Z)$ has Stiefel-Whitney classes and the \emph{Segre}-Stiefel-Whitney classes satisfy
	\begin{equation}\label{eq:pullback}
		f^*\ssu(Z\subset N)=\ssu(f^{-1}(Z)\subset M).
	\end{equation}
\end{theorem}

The proof of this theorem is beyond the scope of this paper, we plan to publish it separately in \cite{FMS}.

\subsection{A geometric meaning of Stiefel-Whitney classes of projective varieties}
The lowest degree CSM and Stiefel-Whitney class of a variety describes its fundamental class, and its top degree term describes its Euler characteristic (modulo 2 for Stiefel-Whitney). Sullivan already asked about the rest of the Stiefel-Whitney classes \cite{sullivan}: ``\emph{One now wonders at the significance of the lower Stiefel Whitney homology classes of these analytic spaces}''. In this section we give a partial answer to the question of Sullivan, based on an analogous result of Ohmoto \cite{Ohmoto2003} and Aluffi \cite{Aluffi2013} for the CSM classes.

\begin{theorem}[Ohmoto-Aluffi] For any complex subvariety $Z$ of the complex projective space $\CP^n$ assign two vectors of integers:
	\begin{itemize}
		\item the coefficients of $[\CP^{i}]$ in the class $\csm(Z\subset \CP^n)$,
		\item the list of Euler characteristics of generic \emph{slices} $Z\cap \CP^i$.
	\end{itemize}
	Then there is a linear transformation---independent of $Z$---of determinant one, which maps the first vector to the second. (The value of $\chi(Z\cap \CP^i)$ is independent of the linear subspace $\CP^i$, assuming that it is in generic position, i.e.\ transversal to a Whitney stratification of $Z$.)
\end{theorem}

Therefore the CSM class of a subvariety $Z$ of $\CP^n$ carries the same information as the Euler characteristics of its intersections with generic projective subspaces. To obtain the Stiefel-Whitney version, we summarize the proof of the Ohmoto-Aluffi theorem. The proof is based on the following transversal pullback property of the Segre CSM classes (the complex analogue of \eqref{eq:pullback}):

\begin{proposition}\cite[Proposition 3.8]{Ohmoto2016} \label{prop:ssm-transversality}
	Let $N$ be a smooth complex algebraic variety and $Z\se N$ be an algebraic subset. Let $f:M\to N$ be an algebraic map to $N$ from a smooth complex algebraic variety $M$, which is transversal to a Whitney stratification of $Z$. Then
	\begin{equation*}
		f^*\ssm(Z\se N)=\ssm(f^{-1}(Z)\se M),
	\end{equation*}
	where
	\[ \ssm(Z\se N)=\frac{\csm(Z\se N)}{c(TN)}.\]
\end{proposition}

Proposition \ref{prop:ssm-transversality} implies that if $N$ is a smooth, compact, complex algebraic variety, $Z\subset N$ is a subvariety and $Y\subset N$ is a smooth subvariety and transversal to a Whitney stratification of $Z$, then
\begin{equation}\label{eq:euler-of-slice}
	\chi(Z\cap Y)=\int\limits_N\frac{\csm(Z \subset N)\csm(Y \subset N)}{c(TN)}.
\end{equation}

Specializing to $Z\subset \C\P^n$ and $Y=\C\P^i$ transversal to some Whitney stratification of $Z$ (which is a generic condition), we can find a reasonably simple formula calculating the Euler characteristics of the slices $Z\cap \C\P^i$  of $Z$. By formula \eqref{eq:euler-of-slice}, there is a linear \emph{Ohmoto-Aluffi transformation} which maps the vector $\csm(Z\se \CP^n)$ to the vector of $\chi(Z\cap \CP^i)$. This transformation is invertible\footnote{This can be shown by specializing to $Z=\CP^j$.}: if we know all the Euler characteristics $\chi(Z\cap \C\P^i)$ then we can calculate $\csm(Z\subset \C\P^n)$.

The real version is the following:

\begin{theorem}\label{thm:real-aluffi} For any real subvariety $Z$ of the real projective space $\RP^n$ assign two vectors of integers:
	\begin{itemize}
		\item the coefficients of $[\RP^i]$ in the class $\su(Z\subset \RP^n)$,
		\item the list of Euler characteristics of the \emph{slices} $Z\cap \RP^i$, for projective subspaces $\RP^i\se \RP^n$ in generic position.
	\end{itemize}
	Then there is a linear transformation---independent of $Z$---of determinant one, which maps the first vector to the second. 
\end{theorem}
\begin{proof}
    A generic projective subspace is transversal to $Z$ by the Kleiman-Bertini theorem. Therefore we can use Sch\"urmann's transversality theorem \cite[(2.17)]{Schurmann2017}: this implies that
	
	\begin{equation}\label{eq:mod2chi}
		\chi_2(Z\cap Y)=\int\limits_N\frac{\su(Z \subset N)\su(Y \subset N)}{w(TN)}.
	\end{equation}
	where $N=\RP^n$ in our case. Using this transversality theorem, the proof of Aluffi \cite{Aluffi2013} can be also applied in the real context of Stiefel-Whitney classes. 

 Alternatively, the Ohmoto-Aluffi transformation on $H^*(\RP^n;\F_2)$ is the same as the Ohmoto-Aluffi transformation on $H^*(\CP^n;\Z)$ mod 2 (with the degrees halved). This follows from \eqref{eq:mod2chi} and $cl_{\R}(c(T\CP^n))=w(T\RP^n)$, where $cl_{\R}$ is Borel and Haefliger's real cycle class map \eqref{eq:clR}, \cite[5.12]{BorelHaefliger1961}. Since the Ohmoto-Aluffi transformation is invertible over $\Z$, it is also invertible over $\F_2$ (in fact it has determinant one). This implies the analogous result for the real projective space and the modulo 2 Euler characteristics. (A similar argument can be carried out for all partial flag varieties.)
\end{proof}

Therefore a partial answer to Sullivan's question is that the  Stiefel-Whitney  classes of a subvariety $Z$ of a projective space contain the same information as the  modulo 2 Euler characteristics of the generic slices of $Z$.
\bigskip

We generalize Theorem \ref{thm:real-aluffi} to study singularity loci of generic smooth maps, which are no longer algebraic.
\begin{proposition} \label{slice}
	Let  $f:\RP^n\to N$ be a generic smooth map, where $N$ is a smooth real algebraic variety, $Z\subset N$ is a subvariety,  and $Y=f^{-1}(Z)$. Then for a generic projective subspace $\RP^i$ with inclusion map $j:\RP^i\to \RP^n$:
	\[\ssu(Y\cap \RP^i\subset \RP^i)=j^*\ssu(Y\subset \RP^n).\]
\end{proposition}
This statement is a consequence of Theorem \ref{thm:generic-pullback}. The proof will be published in \cite{FMS}.
\bigskip

Consequently we can generalize Theorem \ref{thm:real-aluffi}:
\begin{theorem}\label{thm:real-aluffi4almost}Let  $f:\RP^n\to N$ be a generic map, where $N$ is a smooth real variety, $Z\subset N$ is a subvariety,  and $Y=f^{-1}(Z)$. We can assign two vectors of integers to $Y$:
	\begin{itemize}
		\item the coefficients of $[\RP^i]$ in the class $\su(Y\subset \RP^n)$,
		\item the list of Euler characteristics of the \emph{slices} $Y\cap \RP^i$ for projective subspaces $\RP^i\se \RP^n$ in generic position.
	\end{itemize}
	Then there is a universal linear transformation of determinant one, which maps the first vector to the second. 
\end{theorem}

\section{Universal obstructions: singularity loci of smooth maps}
\label{sec:degloci}

%In this section we introduce the .

Loosely speaking a complex contact singularity $\eta$ is a subvariety of the space of holomorphic map germs $(\C^n,0)\to (\C^{n+l},0)$. More precisely, consider a holomorphic finitely $\mathcal{K}$-determined map germ $g:(\C^n,0)\to (\C^{n+l},0)$ and its contact orbit $\eta=\mathcal{K}g$ in the space of map germs. It is useful to study $g$ and its trivial unfoldings at the same time as many of their numerical invariants agree. The \emph{trivial unfolding} $\sigma g :(\C^{n+1},0)\to (\C^{n+l+1},0)$ is defined as
\[ \sigma g(x_1,\dots,x_n,x_{n+1}):=(g(x_1,\dots,x_n),x_{n+1}).\]
If we refer to a singularity or any of its trivial unfoldings, then we use the notation $\eta(l)$, where $l$ denotes the relative codimension.
We will also use the notation $\eta(n+a,n+a+l)$ for the contact orbit of the $a$-fold trivial unfolding of $g$. The case of real singularities is completely analogous.
\bigskip

Our main objective  is to study obstructions for the existence of a smooth map $f:M\to N$ avoiding a certain singularity $\eta$. These results come in two levels: First we give an abstract reasoning that a non-zero cohomology class is supported on the singularity locus $\eta(f)$, proving that $\eta(f)$ cannot be empty. Second, we interpret this obstruction as a non-zero invariant of the singularity locus $\eta(f)$ for \emph{generic} $f$. The case of the fundamental class is well studied, so as a warmup we review it in the next section.

\subsection{Universal fundamental classes of contact singularities -- Thom polynomials} \label{sec:tp} We will use the existence of universal classes or Thom polynomials.
For a complex contact singularity $\eta$ the \emph{Thom polynomial} of $\eta$, $\tp_\eta\in\Z[\ccc]$ has the following universal property:

\begin{theorem}\cite{Thom1955}, \cite{HaefligerKosinski1958}  \label{thm:thom}
	If $f:M\to N$ is a proper holomorphic map of complex manifolds then the substitution
	\[ \tp_\eta(f):=\tp_\eta(c_i\mapsto c_i(f^*(TN)\ominus TM))\in H^*(M;\Z)  \]
	is  supported on $\bar\eta(f)$, the $\bar\eta$-locus of the map $f$ (cf.\ Definition \ref{def:support}). This implies that $\tp_\eta(f)$ is an obstruction for having a map $g:M\to N$ homotopic to $f$ with empty $\bar\eta$-locus.
	%In addition, if $f$ satisfies a certain transversality condition, then
	%\[   \tp_\eta(f)=[\bar\eta(f)\subset M].\]
\end{theorem}

The existence of the Thom polynomial of a contact singularity in terms of quotient variables is related to the fact that all the trivial unfoldings $\eta(n+a,n+a+l)$ have the same Thom polynomial. For more details, see e.g.\ \cite{FeherRimanyi2012}. For a general introduction on contact singularities ant their Thom polynomials, see \cite{primer}.
\bigskip

There is an analogous theory for real analytic map germs. The (modulo 2) Thom polynomial $\tp_{\eta}$ of a real contact singularity $\eta$  is an element of $\F_2[\www]$. An important advantage of the  case of real smooth maps is the existence of generic maps. Therefore we have the following real version of Theorem \ref{thm:thom}:

\begin{theorem}\cite{Thom1955}, \cite{HaefligerKosinski1958}  \label{thm:thom-geom}
	If $f:M\to N$ is a proper smooth map of smooth manifolds and $\eta$ is a real singularity, then the substitution
	\[ \tp_\eta(f):=\tp_\eta(w_i\mapsto w_i(f^*(TN)\ominus TM))\in H^*(M;\F_2)  \]
	is  supported on $\bar\eta(f)$, the $\bar\eta$-locus of the map $f$ (cf.\ Definition \ref{def:support}). This implies that $\tp_\eta(f)$ is an obstruction for having a map $g:M\to N$ homotopic to $f$ with empty $\bar\eta$-locus.
	
	In addition, for generic $f$ we have
	\[   \tp_\eta(f)=[\bar\eta(f)\subset M].\]
\end{theorem}

The first part of this theorem gives an abstract obstruction: If $\tp_\eta(f)\neq0$, then there is no smooth map in the homotopy class of $f$ with empty $\eta$-locus. The second part gives the geometric interpretation of this obstruction: for a generic smooth map in the homotopy class of $f$, the fundamental class of the $\clos{\eta}$-locus is $\tp_\eta(f)$.

There are various methods developed to calculate complex Thom polynomials. A theorem of Borel and Haefliger \cite[Theorem 6.2]{BorelHaefliger1961} allows us to use these results to calculate real Thom polynomials:
\begin{theorem}[\cite{BorelHaefliger1961}] \label{thm:BH}
	Let $\eta_\C$ be the complexification of a real contact singularity $\eta$. Suppose that $\tp_{\eta_\C}=\sum{ a_Ic_I}$. Then  $\tp_{\eta}=\sum{ a_Iw_I}$, where the coefficients $a_I$ are reduced mod 2.
\end{theorem}
\begin{remark}
	We say that $\eta_\C$ is the {\em complexification} of the real singularity $\eta$ if  their analytic closures  are defined by the same equation and $\codim_\C\eta_\C=\codim_\R\eta$.
	%The codimension condition is not always satisfied see \cite{vass-ser}.
	This condition is only satisfied if the analytic closure of a real singularity is given by an algebraic subset of germs. For example the analytic closure of the real $I_{22}$ does not agree with its algebraic closure, but the analytic closure of $I_{22}\cup II_{22}$ does. Part of Borel and Haefliger's theorem is that all real algebraic singularities have a modulo 2 Thom polynomial. But for example $I_{22}$ does not admit a Thom polynomial. Therefore by real contact singularity, we will always mean a real, contact-invariant algebraic set.
\end{remark}

\subsection{Universal Segre-Stiefel-Whitney classes of contact singularities} \label{sec:ssw4eta}
We want to generalize the results of the previous section to Stiefel-Whitney classes.

For  holomorphic maps, there is an enhancement $\ssmeta$ of the Thom polynomial defined by Ohmoto \cite[Thm 4.4]{Ohmoto2016}, called the Segre-Schwartz-MacPherson-Thom polynomial or SSM-Thom polynomial in short. The SSM-Thom polynomial $\ssmeta$ is a power series $\ssmeta\in\Z[[\ccc]]$, whose lowest degree term is equal to the Thom polynomial, and which has the following universal property:

If $f:M\to N$ is a holomorphic map of complex manifolds then the substitution
\[ \ssmeta(f):=\ssmeta(c_i\mapsto c_i(f^*(TN)\ominus TM))\in H^*(M;\Z)  \]
is supported on $\bar\eta(f)$, the $\bar\eta$-locus of the map $f$. This implies that $\ssmeta(f)$ is an obstruction for having a map $g:M\to N$ with empty $\bar\eta$-locus homotopic to $f$. If $f$ satisfies a certain (stronger) transversality condition, then
\[   \ssmeta (f)=\ssm(\eta(f)\subset M).\]

Now we construct the analogous universal polynomial for Segre-Stiefel-Whitney classes.

\begin{theorem} \label{thm:ssw} Let $\eta$ be a real contact singularity, which is $k$-determined for some $k>0$, and such that the $k$-jet $j^k\eta\subset J^k(n,p)$ is algebraic. Then there is a unique power series $\ssu_\eta\in \F_2[[\www]]$ with the following properties:
	\begin{enumerate}
		\item \label{item:support} If $f:M\to N$ is a smooth map of  manifolds then the substitution
		\[ \ssu_\eta(f):=\ssu_\eta(w_i\mapsto w_i(f^*(TN)\ominus TM))\in H^*(M;\F_2)  \]
		is supported on $\clos{\eta}(f)$, the $\clos{\eta}$-locus of the map $f$. This implies that $\ssu_\eta(f)$ is an obstruction for having a map $g:M\to N$ homotopic to $f$ with empty $\clos{\eta}$-locus.
		\item \label{item:stable} The Segre-Stiefel-Whitney class of $\eta$  and its trivial unfolding are the same.
		
		\item \label{item:bh}Suppose that $\eta_\C$ is the complexification of  $\eta$. Then for  $\ssm_{\eta_\C}=\sum{ a_Ic_I}$ we have  $\ssu_\eta=\sum{ a_Iw_I}$.
		\item \label{item:lowest} The lowest degree term of $\ssu_\eta$ is equal to $\tp_\eta$, the (modulo 2) Thom polynomial of $\eta$.
		\item \label{item:trans-pullback} If $f:M\to N$ is a generic smooth map of real algebraic manifolds, then
		\[  \ssu_\eta(f)=\ssu(\eta(f)\subset M).\]
	\end{enumerate}
\end{theorem}
\begin{proof} To construct $\ssu_\eta$ we follow \cite{FeherRimanyi2012} and \cite{Ohmoto2016} and define $\ssu_\eta$ as the $\GL(n)\times\GL(p)$ equivariant Segre-Stiefel-Whitney class of $j^k\eta\subset J^k(n,p)$. To define the equivariant Segre-Stiefel-Whitney class we use finite dimensional approximations of the classifying space $B(\GL(n)\times\GL(p))$:
	$\gr_n(\R^N)\times \gr_p(\R^N)$ for $N\gg0$. We construct a vector bundle $E\to \gr_n(\R^N)\times \gr_p(\R^N)$ with fiber $J^k(n,p)$ and define $\ssu_\eta$ as $\ssu(\eta(E)\subset E)$.
	(\ref{item:support}) and (\ref{item:stable}) follows from the transversality of the trivial unfolding map to the $\mathcal{K}$-orbits exactly as in the complex case. And the same argument (using supersymmetry, see e.g.\ \cite{FeherRimanyi2012}) implies that this polynomial can be written in quotient variables.
	%By \cite[Thm 4.4]{Ohmoto2016} $\ssm_{\eta_\C}$  can be written in quotient variables. (\ref{item:bh}) implies that this holds for $\ssu_{\eta}$ as well.
	
	To prove (\ref{item:bh}) notice that $\ssm_{\eta_\C}$ is defined using the finite dimensional approximation
	$\gr_n(\C^N)\times \gr_p(\C^N)$ for $N\gg0$ of the classifying space $B(\GL(n;\C)\times\GL(p;\C))$ and a vector bundle $E_\C\to \gr_n(\C^N)\times \gr_p(\C^N)$ with fiber $J^k_\C(n,p)$. Then we can apply Theorem \ref{thm:bsy} to $X=E_\C$ and $Z=\eta_\C(E)$.
	
	(\ref{item:lowest}) is a consequence of (\ref{item:bh}) and Theorem \ref{thm:BH}.
	
	For (\ref{item:trans-pullback}) recall now the Thom jet transversality theorem (see e.g.\ \cite[2.3, 2.4]{Damon50}):
	\begin{theorem}\label{thm:thom-jet} For a submanifold of jet spaces
		$W\subset J^k(N, P)$, the set
		\[\mathcal{W} = \{f \in  C^\infty(N, P) : j^k(f) \text{ is transverse to }W\} \]
		is a residual subset of $C^\infty(N, P)$. If $W$ is a closed submanifold then
		$\mathcal{W}$ is open. More generally if $W$ is an open stratum of a closed Whitney stratified
		set, then $\mathcal{W}$ contains an open dense subset.
	\end{theorem}
	
	Now (\ref{item:trans-pullback}) follows from Theorem \ref{thm:generic-pullback}.
\end{proof}

Part (\ref{item:support}) of Theorem \ref{thm:ssw} gives an abstract obstruction: If $\ssu_\eta(f)\neq0$, then there is no smooth map in the homotopy class of $f$ with empty $\eta$-locus. Part (\ref{item:trans-pullback}) gives the geometric interpretation of this obstruction: for a generic smooth map in the homotopy class of $f$ the Segre-Stiefel-Whitney class of the $\eta$-locus is $\ssu_\eta(f)$. Note that since $\ssu$ satisfy the motivic property, the $\ssu_\eta$-classes of the Theorem are also defined when $\eta$ is not closed.

\begin{remark}
	By the Nash-Tognoli theorem \cite{nash}, \cite{tognoli}, every compact smooth manifold is diffeomorphic to a real algebraic variety. So the algebraic condition on $M$ in part (\ref{item:trans-pullback}) can be dropped for $M$ compact.
\end{remark}

\section{The avoiding ideal $\A_\eta$}\label{sec:avoiding}

Elements of the avoiding ideal $\A_\eta$ of a singularity $\eta$ are characteristic classes 
\emph{universally supported} on degeneracy loci.  Therefore these characteristic classes evaluated on the normal bundle $\nu_f$ of a map $f:M\to N$ provides obstructions for the existence of a map $g$, homotopic to $f$, which avoids $\eta$.

The notion of universally supported classes was introduced in \cite{Pragacz1988}, \cite[p.39]{FultonPragacz}. As a generalization for general group actions the notion of avoiding ideal was introduced in \cite{FeherRimanyi2004}, which we briefly review now.

\subsection{Avoiding ideal}

\begin{definition}\label{def:support}
	Let $Z\subset X$ be a closed subset of a topological space\footnote{Although the definition makes sense for arbitrary subsets $Z\subset X$, it is most natural to consider these only for closed subsets.}. An element $x\in H^*(X)$ is \emph{supported on $Z$} if it is in the kernel of the restriction homomorphism:
	\[ r: H^*(X)\to H^*(X\setminus Z). \]	
\end{definition}
The notion of support generalizes to the equivariant setting:
\begin{definition}\label{def:avoiding} Let the Lie group $G$ act on the (real or complex) vector space $V$ and let $Z\subset X$ be a $G$-invariant subvariety. An element $x\in H^*(X)$ is \emph{supported on $Z$} if it is in the kernel of the restriction homomorphism:
	\[ r: H^*_G(V)\to H^*_G(V\setminus Z). \]	
\end{definition}

In the equivariant setting $\Ker(r)$ is also called the \emph{avoiding ideal} of $Z\subset V$, because of the following fact \cite[Theorem 2.3]{FeherRimanyi2004}:
\begin{proposition}\label{prop:avoiding_bundle} Let the Lie group $G$ act on the (real or complex) vector space $V$ and let $Z\subset V$ be a $G$-invariant subvariety. Let $P\to M$ be a principal $G$-bundle and $\kappa:M\to BG$ be the classifying map of $P$. Let $E=P\times_G V$ and $Z(E)=P\times_GZ$ be the associated bundles and $\si:M\to E$ be a section. Then for all $x$ in the kernel of the restriction map
	\[ r: H^*_G(V)\to H^*_G(V\setminus Z) \]
	the class $\kappa^*x$ is supported on the \emph{$Z$-locus} $Z(\si):=\sigma^{-1}\big(E(Z)\big)$. Consequently if $\kappa^*x$ is non-zero, then there is no section $\sigma:M\to E$ \emph{avoiding $Z$}, i.e.\ with empty $Z$-locus.
\end{proposition}

Our primary example is the contact group $G=\mathcal K^k(n,p)$ acting on the jet space $V=J^k(n,p)$, with $Z$ a $k$-determined contact singularity. Then
\[
\ker\big (r:H_G^*(V)\to H^*_G(V\backslash Z)\big) \se H_G^*(V)=\F_2[a_1\stb a_n,b_1\stb b_p].
\]
Despite the simple definition, avoiding ideals are difficult to calculate.

\subsection{Stable avoiding ideal}

The following definition is inspired by the stability property of the Segre-Stiefel-Whitney classes ((\ref{item:stable}) of Theorem \ref{thm:ssw}). Let 
\begin{equation}\label{eq:rhonp}
	\rho_{n,p}:\F_2[\www]\to \F_2[a_1\stb a_n, b_1\stb b_{p}]
\end{equation} 
be the map defined by making the homogeneous terms of the following formal expression equal:
\[
1+w_1+w_2+\ldots =\frac{1+b_1+\ldots +b_{p}}{1+a_1+\ldots +a_n},
\]
where $w_i, a_i, b_i$ have formal degree $i$.

\begin{definition} $\alpha\in\F_2[\www]$ is in the \emph{stable avoiding ideal} $\A_\eta$ of the singularity type $\eta\in J(n,n+l)$, if $\rho_{N,N+l}(\alpha)\in \A(\eta^k(N,N+l))$ for all $N\geq n$ and $k\gg0$. (Here $\eta(N,N+l)=\mathcal{K}.\tau(\eta)$ is the contact orbit of the trivial unfolding $\tau:J(n,n+l)\to J(N,N+l)$ of $\eta$).
	%\[\A^s_\eta:=\bigcap_{m\to \infty}\rho_{m,m+l}^{-1}(\A(\eta(m, m+l))) \]	
\end{definition}

The complex version can be defined analogously. The elements of the stable avoiding ideal are obstructions to the existence of a smooth map $f$ with empty $\eta$-locus. Indeed, $j^kf$ is a section of the jet-bundle $J^k(M,N)$, so we can apply Proposition \ref{prop:avoiding_bundle} to get:
\begin{proposition}\label{prop:avoiding_is_obstruction}
	Let $\eta\in J^k(n,p)$ be a real contact singularity and $\al\in \A_\eta$ be an element of the stable avoiding ideal. Let $f:M^m\to N^n$ be a smooth map between real manifolds. If $$\al(w_1(f),\ldots,w_m(f))\neq 0,$$
	then $\eta(f)\neq \emptyset$. Here $w_i(f)=w_i(f^*TN\ominus TM)$, the $i$th Stiefel-Whitney class of the virtual normal bundle of $f$.
\end{proposition}

Stable avoiding ideals are only known in few cases---notably, the avoiding ideal of the $A_2$ singularity in the real case mod 2 \cite[Theorem 1]{Terpai} and the $\Sigma^i$ in the complex case \cite[Theorem 1.1]{Pragacz1996}, \cite[Theorem 4.2]{FultonPragacz}. 

Given a partition $\la=(\la_1\stb \la_r)$, for Schur polynomials $s_\la(w_1,w_2,\ldots )\in \F_2[w_1,w_2,\ldots]$ (or in $\Z[c_1,c_2,\ldots]$) we fix the following convention:
\begin{equation}\label{eq:schurdef}
	s_\la(w)=\det(w_{\la_i+j-i})_{i,j=1\stb r}
\end{equation}
\begin{theorem}[Pragacz]\label{thm:FultonPragacz}
	For $i+l\geq0$ the stable avoiding ideal of $\clos{\Si^i}(l)$ is
	\[
	\A_{\clos{\Si^i}(l)}=\Z\bra s_{\la}:(i+l)^i\se \la\ket,
	\]
	where $\bra\ \ket$ denotes the generated  $\Z$-module.
\end{theorem}
The analogous result holds over mod 2:
\begin{theorem}\label{thm:mod2avoiding-sigma-i}
	For $i+l\geq0$ the modulo 2 stable avoiding ideal of $\clos{\Si^i_\R}(l)$ is
	\[
	\A_{\clos{\Si^i_\R}(l)}=\F_2\bra s_{\la}:(i+l)^i\se \la\ket,
	\]
	where $\bra\ \ket$ denotes the generated $\F_2$-module.
\end{theorem}
\begin{proof} One can  repeat the proof of  \cite[Theorem 1.1]{Pragacz1996} modulo 2, or modify the calculation given in \cite[Sec. 6]{FeherRimanyi2004}.
\end{proof}

\begin{theorem}[Terpai] \label{thm:Terpai}
	For $l\geq0$ the stable avoiding ideal of $\clos{A_2(l)}$ is
	\[\A_{\clos{A_2}(l)}=\F_2\bra s_\la: (l+1)^2\se \la\ket,\]
	where $\bra\ \ket$ denotes the generated $\F_2$-module.
\end{theorem}
Terpai gives ideal generators of $\A_{\clos{A_2}(l)}$ but it is easy to see that it is equivalent to the description above.
\begin{remark} \label{rem:avoidwithrestr}
	Terpai's result can be obtained by the restriction equation method as well, since the normal Euler condition of \cite[Definition 3.3]{FeherRimanyi2004} holds. Similarly, for the complex cusps the stable avoiding ideal can be given as the kernel of a ring homomorphism, but a generating system has not been calculated yet. This is the full list of known stable avoiding ideals of singularities and it is not clear if any other ones can be calculated with the present methods.
\end{remark}

In these cases the image of the stable avoiding ideal under $\rho_{n,n+l}$ generates the avoiding ideal of $\eta^k(n,n+l)\subset J^k(n,n+l)$, see Pragacz \cite[Theorem 1.1]{Pragacz1996}. We don't expect this to hold for general singularities. Also we have a Borel-Haefliger type connection between the stable avoiding ideal of the complex and the real $\Sigma^i$'s: the second one is obtained by the substitution $c_i\mapsto w_i$. It would be interesting to see which other singularities share this property.

Comparing Theorems \ref{thm:mod2avoiding-sigma-i} and \ref{thm:Terpai} we obtain
\begin{corollary} \label{cor:a2-si2}
	For $l\geq0$ the following stable avoiding ideals are equal:
	\begin{equation}\label{eq:A2S2} \A_{\clos{A_2}(l)}=\A_{\clos{\Si^2}(l-1)}. \end{equation}
\end{corollary}

\begin{remark} \label{rmk:si2-1}
	Notice that Theorems \ref{thm:FultonPragacz} and \ref{thm:mod2avoiding-sigma-i} are valid for $\Si^2(-1)$, cf.\ Remark \ref{rem:duality}. Allowing the $l=-1$ case for $\Si^2$ will make our discussion easier.
\end{remark}

\subsection{Finding elements of the stable avoiding ideal}
Finding the stable avoiding ideal of other singularities is a challenge in general, but we have methods to find elements in it. The Thom polynomial $\tp_{\eta(l)}$ is in the stable avoiding ideal of $\A_{\eta(l)}$---indeed, $\tp_{\eta(l)}$ is \quot{the lowest degree element} of $\A_{\eta(l)}$, essentially by definition. Similarly, the homogeneous terms of the Segre-SW class $\ssu_{\eta(l)}$---see Section \ref{sec:ssw4eta}---are also elements of the stable avoiding ideal by Theorem \ref{thm:ssw} (\ref{item:support}). The following is also an immediate consequence of the definition of the avoiding ideal:
\begin{proposition}\label{prop:avoiding_complicated} If the singularity $\xi$ is in the closure of $\eta$, then $\A_{\xi}\subset \A_{\eta}$, in particular $\tp_\xi\in \A_{\eta}$. \end{proposition}

Further elements of the stable avoiding ideal can be constructed using the lowering operators $\flat(i)$ introduced in \cite{FeherRimanyi2007}. First we recall another variant of unfolding, the \emph{zero-unfolding}:
\begin{definition} For  a holomorphic map germ $g:(\C^n,0)\to (\C^{n+l},0)$ the zero-unfolding  $\delta g: (\C^{n},0)\to (\C^{n+l+1},0)$ is defined as
	\[ \delta g(x_1,\dots,x_n):=(g(x_1,\dots,x_n),0).\]
\end{definition}
Although the main invariant---the quotient algebra---of $\delta g$ is the same as of $g$, its codimension increases \cite{Mather2012}. Given a singularity $\eta(k)$ of codimension $k$, by repeatedly zero-unfolding  $\eta(k)$, one obtains a singularity $\eta(l)$ with relative codimension $l\geq k$. There is a strong connection between the stable avoiding ideals of $\eta:=\mathcal{K}g$ and $\eta^\sharp:=\mathcal{K}(\delta g)$. To understand this we recall the notion of lowering operators from \cite{FeherRimanyi2007}:
\begin{definition} Let $p(c_1,c_2,\dots,c_i)$ be a polynomial. Then $p^{\flat(j)}$ is defined as the coefficient of $t^j$ in $p(c_1+t,c_2+tc_1,\dots,c_i+tc_{i-1})$.
\end{definition}
One of the key results of \cite{FeherRimanyi2007} is the following:

\begin{theorem}\label{thm:lowering} If $p\in \A_{\eta^\sharp}$ then $p^{\flat(j)}\in \A_{\eta}$ for $j\geq0$. \end{theorem}

For several singularities $\eta$ the Thom polynomial $\tp_{\eta(l)}$ is known for all $l$, see \cite{FeherRimanyi2012} and \cite{Kazarian2017}. Therefore we can obtain several elements in $\A_{\eta(l-k)}$ by applying the various lowering operators $k$ times. Using these lowering operators seems to be the most prolific method to obtain elements in the stable avoiding ideal. There are other methods: in the complex case we can use restriction equations for low degree (\cite{FeherRimanyi2004}), and in the real case we can use the fact that $\A_{\eta}$ is closed under the Steenrod operations.

\begin{remark} \label{flat0} The operator $\flat(0)$ is the identity, however Theorem \ref{thm:lowering} is useful even in this case: $\A_{\eta^\sharp}\subset \A_{\eta}$. The proof of this case is immediate.
\end{remark}

\section{Obstructions for fold and Morin maps}\label{sec:obstructions_fold_morin}
In this section we give lower bounds for the fold numbers and Morin numbers of real projective spaces. First we give some general definitions:
\subsection{Lower bounds for the $\eta$-avoiding number}
For a smooth manifold $M$, we will denote
\begin{equation}\label{eq:wb}
	\bar{w}(M)= \frac{1}{w(M)}.
\end{equation}
\begin{definition} Let $M$ be a compact $n$-dimensional smooth manifold and $\eta$ a contact singularity. Then we call  the smallest $k\geq 0$ such that there is a smooth map from $M$ to $\R^{n+k}$ with no $\eta$-points the \emph{$\eta$-avoiding number} $\alpha(M,\eta)$.
	For $\eta=A_2$ we 
	%also use the notation $\varphi(M):=\alpha(M,A_2)$ and 
	call it the \emph{fold number}, and for  $\eta=\Sigma^2$ we 
	%also use the notation $\mu(M):=\alpha(M,\Sigma^2)$ and 
	call it the \emph{Morin number} of $M$.
\end{definition}
The case of $\eta=\Sigma^1$ ---immersions--- has been extensively studied, quite famously in \cite{Cohen1985}; see \cite{Davis} for a compilation of many such results and references therein. In this paper we make  steps to expand these results to other singularities and give lower bounds for the fold and Morin numbers.

\bigskip

An obvious upper bound for $\alpha(M,\eta)$ can be obtained from the observation that if the expected dimension of the locus is negative then for generic $f$ the locus is empty. A simple lower bound for the $\eta$-avoiding number can be given using the Thom polynomials $\tp_{\eta(l)}$ (see Section \ref{sec:tp}):
\begin{definition}
	Let $M$ be a compact $n$-dimensional smooth manifold and $\eta$ a contact singularity. Then
	\[ \tau(M,\eta):=\min \{l\geq l_0: \tp_{\eta(j)}(\bar w(M))= 0 \text{ for all } j\geq l \},\]
	where $l_0$ is the minimal relative codimension where $\eta$ can appear.
\end{definition}
\begin{remark}
	For example for $A_2$ we have $l_0=0$, for $III_{23}$  we have $l_0=1$ and for $\Sigma^2$ we have $l_0=-1$ and more generally for $\Sigma^i$, $l_0=-i+1$. For contact singularities it is more natural to treat the $l\geq 0$ and $l\leq 0$ cases separately. For instance, $A_2(l)$ is also defined for $l<0$, but $A_2(l<0)$ are not related to the $l\geq 0$ cases via zero-unfolding. In this paper, we focus on the $l\geq 0$ cases, the only $l<0$ case that we will consider is the case of $\Si^2(-1)$, cf.\ Remark \ref{rmk:si2-1}.  %Strictly speaking $\Sigma^2(-2)$ also makes sense but it has codimension 0 with Thom polynomial equal to 1, so it is not very interesting. 
\end{remark}

Then we have $\tau(M,\eta)\leq \alpha(M,\eta)$, since the Thom polynomial is universally supported on the $\eta$-locus. The Thom polynomials of the simpler singularities are well-known so this is a well-calculable lower bound, although not very strong. A stronger bound can be given using the stable avoiding ideal:

\begin{definition}
	Let $M$ be a compact $n$-dimensional smooth manifold and $\eta$ a contact singularity. Then
	\[ \kappa(M,\eta):=\min \{l\geq l_0: \ga(\bar w(M))= 0 \text{ for all } \ga\in \mathcal{A}_{\eta(l)} \},\]
	where $l_0$ is the minimal relative codimension where $\eta$ can appear.
\end{definition}

By the definition of the avoiding ideal and using Proposition \ref{prop:avoiding_is_obstruction} and Remark \ref{flat0} we obtain the lower bound $\kappa(M,\eta)\leq \alpha(M,\eta)$. Since the Thom polynomial is in the avoiding ideal we have $\tau(M,\eta)\leq \kappa(M,\eta)$. The disadvantage of this stronger bound is that the avoiding ideals are only known in special cases. However we can define several elements of the avoiding ideal and obtain stronger bounds than the Thom polynomial bound. Several examples will be given in the rest of the paper.

\subsection{Lower bounds for the fold and Morin numbers}

In the case of $\eta=A_2$ and $\eta=\Sigma^2$ the Thom polynomials and the stable avoiding ideals are known, so calculating the lower bounds $\tau(M,\eta)$ and $\kappa(M,\eta)$ is possible. In fact it is enough to calculate one case (e.g the fold case) and we obtain the other one. Indeed, by Corollary \ref{cor:a2-si2} the avoiding ideals are the same (with a shift in $l$), and the same is true for the Thom polynomials:
\[\tp(A_2(l))=\tp(\Sigma^2(l-1))=s_{l+1,l+1}.\]
The case of $A_2$ for real projective spaces $\RP^n$ was calculated by Terpai in \cite{Terpai}. By the above remark this leads to the calculations for the $\eta=\Sigma^2$ case.

In this section we reprove parts of Terpai's results because we want to use the calculations later to find a geometric meaning for the bounds $\kappa(\RP^n,A_2)$ and $\kappa(\RP^n,\Sigma^2)$.

The answer depends on the binary expansion of $n$, and whether two consecutive 1's occur in the binary expansion. Suppose that $n$ has $t$ digits, i.e. $2^{t-1}\leq n<2^{t}$. Let $L(t)$ denote the largest $t$-digit number without two consecutive 1's, i.e $L(t)=101010\dots=\lfloor2^{t+1}/3\rfloor$. Then Terpai's results can be summarized as follows.

\begin{definition}\label{def:terpai}
	Suppose that $n$ has $t$ digits, then it falls  into exactly one of the following four classes:
	\begin{itemize}
		\item[a)] there are no consecutive 1's in the binary expansion of $n$ and $n$ is even,
		\item[b)] there are no consecutive 1's in the binary expansion of $n$ and $n$ is odd,
		\item[c)] there are consecutive 1's in the binary expansion of $n$ and $L(t)<n$,
		\item[d)] there are consecutive 1's in the binary expansion of $n$ and $n < L(t)$.
	\end{itemize}
	Let $p$ be the largest number such that $2^p$ divides $n$.
	
	In case d), let $u$ be the largest number such that the binary expansion of $n$ is of the following form (this always exists in case d)):
	\[n=\underbrace{[n_{t-1}\stb n_{u+3}]}_{a}[0,1,1]\underbrace{[n_{u-1}\stb n_0]}_{b}\]
	Therefore $n=2^u\cdot (8a+3)+b$, and the binary expansion of $a$ does not contain consecutive 1's.
	Define $\ka$ to be the following function:
	\begin{equation}\label{eq:terpai}
		\ka(n)=\begin{cases}
			\frac{n-2^p}{2}+1\qquad &\text{in case a)}\\
			\frac{n-1}{2}\qquad &\text{in case b)}\\		
			2^{t}-n-1\qquad&\text{in case c)}\\
			2^{u+2}a+(2^u-1-b)\qquad &\text{in case d)}		
			%		2^u\cdot (12a+5)-n+1	
		\end{cases}
	\end{equation}
\end{definition}
\begin{theorem}\cite{Terpai}\label{fold-kappa} For all $n\geq1$
	\[\kappa(\RP^n,A_2)=\ka(n).\]
\end{theorem}
\begin{corollary} \cite{Terpai}\label{fold-obstructions}
	If there exists a fold map $\mathbb{R} \P^n \rightarrow \mathbb{R}^{n+k}$, then $k\geq \kappa(n)$.
	\begin{remark}\label{rmk:terpai}
		Compared to Terpai's theorem \cite[Theorem 3]{Terpai}, there is a $+1$ correction term added in \eqref{eq:terpai}. \cite[Theorem 3]{Terpai} only claims Corollary \ref{fold-obstructions} but his proof shows that Theorem \ref{fold-kappa} is also true. We modified Terpai's cases to adjust them to our arguments.
	\end{remark}

\end{corollary}

\begin{table}
	\begin{center}
		\begin{tabular}{|c|c|c |c| c| c| c| c | c| c| c| c| c| c|c|c|c| c| c| c|c|c|}
			\hline
			$n$ &$\textcolor{red}{4_a}$ & $5_b$ & $6_c$ &$7_c$ &$\textcolor{red}{8_a}$ & $9_b$ & $10_a$ &  $11_c$ &$12_c$ &$13_c$ &$14_c$ &$15_c$ & $\textcolor{red}{16_a}$ & $17_b$ & $18_a$ & $19_d$	&$\textcolor{red}{20_a}$ & $21_b$ \\ \hline
			%						$[n]_2$ &$100$ & $101$ & $1000$ & $1001$ & $1010$ &  $1011$ & $10000$ & $10001$ & $10010$ & $10011$	&$10100$ & $10101$ \\ \hline
			$\tau(\RP^n,A_2)$ & $\textcolor{red}{0}$ &2&1&0 &$\textcolor{red}{0}$ & 4& 5& 4&3&2&1& 0& $\textcolor{red}{0}$& 8& 9& 8& \textcolor{red}{8} &  10   \\ [1ex] \hline
			$\kappa(\RP^n,A_2)$ & $\textcolor{red}{1}$&
			$2$&1&0& $\textcolor{red}{1}$&$4$  & $5$ & $4$&3&2&1&0  &$\textcolor{red}{1}$&$8$  & $9$ & $8$ & $\textcolor{red}{9}$ & $10$\\[1ex] \hline
		\end{tabular}
	\end{center}
	\caption{The Thom polynomial and the avoiding ideal bounds for fold maps. The columns where the two bounds are different are colored in red. %- see Remark \ref{rmk:terpai}
	}
	\label{table:Terpai}
\end{table}

Using Corollary \ref{cor:a2-si2} we immediately obtain %(cf.\ Remark \ref{rem:a2-si2})
\begin{theorem} For all $n\geq1$  %and $n\neq2^p$
	\[\kappa(\RP^n,\Sigma^2)=\ka(n)-1.\]
\end{theorem}
Notice that $l=-1$ is allowed for $\Sigma^2$.
\begin{corollary}\label{thm:morin-obstructions}
	If there exists a Morin map $\mathbb{R} \P^n \rightarrow \mathbb{R}^{n+k}$, then $k\geq \kappa(n)-1$.
\end{corollary}

These bounds are slightly better then the Thom polynomial bounds, see Table \ref{table:Terpai} for the first couple of values. In general we have the following Theorem.
\begin{theorem}\label{k=t+1} For  the case a), $p\geq2$ we have
	\[\tau(\RP^n,A_2)= \kappa(\RP^n,A_2)-1=\kappa(n)-1,\]
	and for all the other cases we have
	\[ \tau(\RP^n,A_2)=\kappa(\RP^n,A_2)=\kappa(n).\]
	Similarly for the Morin case in  case a), $p\geq2$  we have
	\[\tau(\RP^n,\Sigma^2)+1 =\kappa(\RP^n,\Sigma^2)=\kappa(n)-1,\]
	and for all the other cases we have
	\[ \tau(\RP^n,\Sigma^2)=\kappa(\RP^n,\Sigma^2)=\kappa(n)-1.\]
\end{theorem}

In Section \ref{sec:eulerchar} we will give a geometric meaning to the extra obstruction in case a) of the Morin case, measuring the Euler characteristics of the $\Sigma^2$-locus.

\begin{remark}The obvious upper bound for the fold number by dimension considerations is $\lceil\frac{n-1}{2}\rceil$, which coincides with the lower bounds for case b), i.e.\ it gives a sharp result. In the other cases there is a gap.
\end{remark}
\subsection{Proof of Theorem \ref{k=t+1}} For the results on $\kappa(\RP^n,A_2)$, see \cite{Terpai}. The analogous results for $\Sigma^2$ are implied by Corollary \ref{cor:a2-si2} on the equality of the two avoiding ideals.

Before proving the Thom polynomial bounds $\tau(\RP^n,A_2)$, we compute some inverse Stiefel-Whitney classes of $\RP^n$. These computations rely on the following general form of the inverse Stiefel-Whitney class of $\RP^n$.

\begin{proposition}\label{prop:invSW}
	Let $n_r\stb n_0$ be the binary representation of a number $n$. Then the inverse of the total Stiefel-Whitney class of $\RP^n$, $\bw(\RP^n):=1/w(\RP^n)$ is
	\begin{equation}\label{eq:inversesw}
		\bw(\RP^n)=(1+x)^m=\sum_{k:k_j\cdot n_j=0,\forall j}x^k.
	\end{equation}
	in $H^*(\RP^n;\F_2)$, where $m=2^{r+1}-1-n$.
\end{proposition}
\begin{proof}
	First, we prove the left hand side: $w(\RP^n)=(1+x)^{n+1}$, so the product equals
	\[(1+x)^{n+m+1}=(1+x)^{2^{r+1}}=1+x^{2^{r+1}}=1\] in $H^*(\RP^n;\F_2)$, which proves the first equality. The second one follows from the next proposition. Notice that $m$ is the \quot{complement} of $n$: in the binary representation of $n$ we interchange the zeros and $1$'s.
\end{proof}
\begin{proposition}\label{prop:binomial}
	%	\label{cor:binomial}
	Given a nonnegative integer $n$, let $n_r\stb n_0$ denote its base 2 representation. Let  $m=2^{r+1}-n-1$. Then for all $k$:
	$$\binom{m}{k}\equiv \begin{cases}
		1\mod 2,\quad \text{if }k_j\cdot n_j=0,\forall j\\
		0 \mod 2,\quad\text{else.}
	\end{cases}$$
\end{proposition}
\begin{proof}
	Lucas's theorem states that
	\begin{equation}\label{eqn:lucas}
		\binom{m}{k}\equiv \prod_{j=0}^r \binom{m_j}{k_j}\mod 2.
	\end{equation}
	Therefore $\binom{m}{k}=1$ if and only if $m_j=0$ implies $k_j=0$. Since $m_j+n_j=1$ for all $j$, this implies the statement.
\end{proof}

\begin{example}
	Take $n=437$, whose binary expansion is $110110101$. Then Proposition \ref{prop:invSW} states that $1/w(\PP^n)$ is the sum of all $x^j$, such that the binary expansion of $j$ is of the form $00*00*0*0$, where $*\in \F_2$ are arbitrary:
	\[ 1/w(\PP^{437})=1+x^2+x^8+x^{10}+x^{64}+x^{66}+x^{72}+x^{74}. \]
\end{example}

\begin{corollary}\label{cor:interval}Let $n$ be such that there are no consecutive 1's in the binary expansion of $n$ and $n$ is even, and let $p$ be the largest number such that $2^p$ divides $n$. Then for $\al=n/2-2^{p-1}$ we have $\bar w_i(\R\P^n)=1$ for $i=\al,\dots,\al+2^p-1$, and $\bar w_{\al+2^p}(\R\P^n)=0$.
	
\end{corollary}
\begin{proof}The binary expansion of $\al$ is obtained by shifting the binary expansion of $n$ to the right, and changing the last 1 into 0. Since $n$ does not contain two consecutive 1's, there are no 1's at the same location in  the binary expansion of $n$ and $\al$. The last $p$ digits of both $n$ and $\al$ are zero, therefore there are no 1's at the same location in  the binary expansion of $n$ and $\al+i$ for $i=1,\dots,2^p-1$ either, implying the first part of the statement. The $p$'th digit of $n$ and $\al+2^p$  are both 1, implying that $\bar w_{\al+2^p}(\R\P^n)=0$.
\end{proof}

Now we establish the Thom polynomial bounds $\tau(\RP^n,A_2)$ stated in Theorem \ref{k=t+1}. In this proof we will write $\bw_i=\bw_i(\RP^n)$. Corollary \ref{cor:interval} implies that for $p\geq2$ the Thom polynomial $\tp_{\Si^2(l-1)}=w_{l+1}^2+w_{l+2}w_{l}$ for $l=n/2-2^{p-1}-1$ does not vanish and that $\tp_{\Si^2(j)}(f)=0$ for $j\geq l$ when evaluated on the inverse Stiefel-Whitney class of $\RP^n$. Notice the different behaviour for $p=1$: then $\al=n/2-1$ and $\al+2^p-1=n/2$, implying that
\[\bar{w}^2_{n/2}+\bar{w}_{n/2-1}\bar{w}_{n/2+1}=x^n.\]
This proves Theorem \ref{k=t+1} for case a).

%For $n=2^p$ both $\kappa(\RP^n,\Sigma^2)$ and $\tau(\RP^n,\Sigma^2)$ is zero, which can be checked directly.

Assume that $n$ is of type b). Since the binary expansion of $n$ does not contain consecutive ones, there is no overlap between $(n-1)/2$ and $n$, so $\bw_{l+2}=x^{l+2}$ by Proposition \ref{prop:invSW}. Since $l+3$ and $n$ are both odd, $\bw_{l+3}=0$.

Assume that $n$ is of type c). The binary expansions of $n$ and $l+2=(2^{s+1}-1)-n$ are disjoint, so $\bw_{l+2}=x^{l+2}$ by Proposition \ref{prop:invSW} and $\bw_{l+3}=0$.

Assume that $n$ is of type d). Write $n=2^u\cdot (8a+3)+b$, i.e.\ the binary expansion:
\[n=\underbrace{[n_s\stb n_{u+3}]}_{a}[0,1,1]\underbrace{[n_{u-1}\stb n_0]}_{b}\]
where $a$ contains no consecutive 1's. We will show $\bw_{l+2}=x^{l+2}$ using Proposition \ref{prop:invSW} by showing that there are no carries when adding  $n$ and $l+2=2^{u+2}a+2^u-b-1$. There are clearly no carries when adding $2^u-b-1$ and $b$. Since there are no consecutive ones in $a$, there are also no carries in the sum $2^{u+2}\cdot a$ and $2^{u+3}\cdot a$, so $\bw_{l+2}=x^{l+2}$.  Since there is a carry in the sum of $l+3$ and $n$ (since there is a carry in the sum of $b$ and $2^u-b$), $\bw_{l+3}=0$. This finishes the proof of Theorem \ref{k=t+1}. \qed

\subsection{All obstructions in the avoiding ideal}
Corollary \ref{cor:interval} can be used to explicitely calculate all the non-vanishing obstructions. We will use this result in Section \ref{sec:Sigmai} to give a geometric interpretation of these obstructions using the Parusi\'nski-Pragacz formula.

\begin{proposition} \label{prop:avoiding_obstructions} Assume that $n$ belongs to case a): there are no consecutive 1's in the binary expansion of $n$ and $n$ is even.
	
	Let  $l=n/2-2^{p-1}-1$ and $s_\lambda \in \A_{\Si^2(l)}$. Then for $\bw=1/w(\P^n)$, $s_\lambda(\bw)\neq0$ if and only if
	\begin{equation}\label{eq:lambda}
		\la=(l+2+p, l+2,1^{q-p})
	\end{equation}
	for $q=n-2(l+2)$ and $p=0,\dots,q$. In these cases $s_\lambda(\bw)=x^n$.
\end{proposition}

We illustrate the general scheme of the computation on the smallest example.
\begin{example}\label{ex:Morinavoiding}
	The first case of Theorem \ref{k=t+1} is $n=20$, with binary expansion $10100$. It states that there is no Morin map $\RP^{20}\to \R^{27}$. This example was discovered by Brandyn Lee and was communicated to us by Rich\'ard Rim\'anyi.
	
	The inverse Stiefel-Whitney class $\bar w(\PP^{20})=1/w(\PP^{20})$ is
	$$\bar w(\PP^{20}) = 1+ x + x^2+ x^3+x^8+x^9+ x^{10} + x^{11}.$$
	Therefore the Thom polynomials $\tp_{\Si^2(j)}=\bw_{j+2}^2+\bw_{j+1}\bw_{j+3}$ vanish for $j\geq 7$, and
	$$\tp_{\Si^2(6)}(\bw(\P^{20}))=\bw_8^2+\bw_7\bw_9=x^{16}.$$
	On the other hand, for $\la=(11,9)$:
	$$ s_\la(\bw(\PP^{20}))=\bw_{11}\bw_9+\bw_{12}\bw_8=x^{20}$$
	is a nonzero element in the avoiding ideal of $\Si^2(7)$.
\end{example}
\begin{proof}[Proof of Proposition \ref{prop:avoiding_obstructions}]
	
	Let $\la=(\la_1,\la_2,\dots,\la_k)$ be a partition of length $k$ (i.e.  $\la_k>0$) such that $s_\lambda \in \A_{\Si^2(l)}$. Then by Theorem \ref{thm:mod2avoiding-sigma-i} we have
	$(l+2,l+2)\se \la$. In the remainder of the proof, we will treat  $s_\lambda(\bw)=\det(\bar{w}_{\lambda_i-i+j})$ ($\bw=1/w(\P^n)$) as the determinant of a $0$-$1$ matrix, since $s_\la(\bw)$ is a homogeneous element in $\F_2[x]/(x^{n+1})$, so
	\[
	s_\la(\bw)=s_\la(\bw)|_{x=1}\cdot x^{|\la|}.
	\]
	%	in \[s_\la(w)\in H^{|\la|}(\PP^n;\F_2)=\F_2[x]/(x^{n+1}),\]
	First, we show that if $\la$ is not of the form \eqref{eq:lambda}, then $s_\la(\bw)=0$. We can assume that $|\la|\leq n$, otherwise $ x^{|\la|}=0$. We study the first two rows:
	\begin{equation}\label{eq:tworows}
		\begin{tabular}{ c c c }
			$ \bar w_{\lambda_1}$&$\cdots$ &$\bar w_{\lambda_1+k-1}$  \\
			$\bar w_{\lambda_2-1}$&$\cdots $&$\bar w_{\lambda_2+k-2}$,
		\end{tabular}	
	\end{equation}
	and see whether these rows are linearly dependent (as elements of $\F_2^{k}$).
	By Corollary \ref{cor:interval},
	\begin{equation}\label{eq:wbar}
		\begin{split}
			\bar{w}_{r}&=1,\qquad r=n/2-2^{p-1},\ldots, n/2+2^{p-1}-1, \\
			\bar{w}_{n/2+2^{p-1}}&=0
		\end{split}
	\end{equation}
	We claim that all elements of \eqref{eq:tworows} are in the range of \eqref{eq:wbar}, and that $\la_1+k-1=n/2+2^{p-1}$ iff $\la$ is of the form \eqref{eq:lambda}. $\la_1$ and $\la_2-1$ are at least $l+1=n/2-2^{p-1}$, so we need to check only for the largest elements:
	\[n\geq |\la|\geq \la_1+\la_2+k-2\geq \la_1+(l+2)+k-2,\]
	implying that
	\[\la_1+k-1\leq n-l-1=n/2+2^{p-1},\]
	by the choice of $l$.

	These estimates and \eqref{eq:wbar} imply that both rows are constant 1 (and therefore $s_\la(\bw)=0$), unless $\la$ is of the form \eqref{eq:lambda}.%, together with $\la_1\geq \la_2-1$ and $\la_2+k-2\leq \la_1+k-1$
	
	It remains to calculate the determinant in these cases. For all choices of $j$ we obtain a matrix $M$ with $M_{ab}=1$ if and only if $b\geq a-1$, except that the top right entry is 0. It is an easy exercise to check that the determinant of these matrices modulo 2 is always one (expand using the first column), which implies our claim.
\end{proof}
\begin{remark} As we saw in Corollary \ref{cor:a2-si2} $\A_{\clos{\Si^2}(l)}=\A_{\clos{A_2}(l+1)}$, so we also understand which elements of $\A_{\clos{A_2}(l+1)}$ evaluate to non-zero classes in case a).
\end{remark}

\section{Euler characteristics of the $\Sigma^r$ loci}\label{sec:Sigmai}
%In this second part of the paper, we turn to applications: we apply formulas for universal Segre-Stiefel-Whitney classes of singularity loci to show that maps $\RP^n\to \R^m$ must have certain types of singularities.
In the previous section we showed that for certain values of $n$ (case a) of Definition \ref{def:terpai}) the avoiding ideal gives obstructions besides the Thom polynomials for the existence of Morin maps of $\RP^n$ to $\R^{n+l}$. In this section
 we interpret these obstructions as the Euler characteristics of the singularity loci using the  Parusi\'nski-Pragacz formula, which calculates the Segre-Schwartz-MacPherson classes of $\Si^r$-loci.
\subsection{The Parusi\'nski-Pragacz formula} 
As a first application of Theorem \ref{thm:ssw}, we can compute the Segre-Stiefel-Whitney classes of the $\Sigma^s(l)$ loci. The Segre-Schwartz-MacPherson classes of $\Si^r$-loci were established by Parusi\'nski-Pragacz \cite{ParusinskiPragacz}, see also \cite{FeherRimanyi2018}:
\begin{theorem}[\cite{ParusinskiPragacz}]\label{thm:PP}
	For $m\leq n$, the $\ssm$ Thom polynomial of $\Si^r\se \Hom(\C^m,\C^n)$ is
	$$ \ssm(\Sigma^r)=\sum_{s=r}^m(-1)^{s-r}\binom{s}{r}\Phi_{m,n}^s,\qquad \ssm(\overline{\Sigma}^r)=\sum_{s=r}^m(-1)^{s-r}\binom{s-1}{r-1}\Phi_{m,n}^s$$
	where $\Phi^s_{m,n}=\rho_{m,n} \Phi^s(n-m)$, where $\rho_{m,n}$ was defined in \eqref{eq:rhonp} and 
	\begin{equation}\label{eq:phisl}
			 \Phi^{s}(l)=\sum_{l(\mu)\leq s}\sum_{l(\nu)\leq s}(-1)^{|\mu|+|\nu|}D_{\mu,\nu}^{s,s+l}s_{(s+l)^s+\mu, \nu^T}
	\end{equation}
	where $s_\la$ denotes the Schur polynomial corresponding to the partition $\la$, (see \eqref{eq:schurdef} for the notational convention), $l(\mu)$ is the length of the partition $\mu$, and
	 \begin{equation}\label{eq:Dmunu}
	 	D_{\mu,\nu}^{s,t}=\det 	\binom{\mu_i+s-i+\nu_j+t-j}{\mu_i+s-i}_{i,j=1,\ldots, s}
	 \end{equation}

\end{theorem}
\begin{remark} \label{rem:duality} For the $m>n$ case, we can use the isomorphism $\Hom(\C^m,\C^n)\iso \Hom((\C^n)^\vee,(\C^m)^\vee)$ to obtain that for $r\geq m-n$
\[ \ssm(\Sigma^r(n,m))=\ssm(\Sigma^{r-m+n}(n,m))^\vee \]
where $(s_\la)^\vee=s_{\la^\vee}$ and $\la^\vee$ denotes the transpose of the partition $\la$, e.g. $(3,1)^\vee=(2,1,1)$.
\end{remark}
By Theorem \ref{thm:ssw} (\ref{item:bh}), we get the following formulas for the Segre-SW classes of the $\Si^r$-loci:
\begin{theorem}\label{thm:sigmarsullivan}
	$$ \ssu(\Sigma^r)=\sum_{s=r}^m\binom{s}{r}\Phi_{m,n}^s,\qquad \ssu(\overline{\Sigma}^r)=\sum_{s=r}^m\binom{s-1}{r-1}\Phi_{m,n}^s$$
\end{theorem}
We will apply these formulas to compute Euler characteristics of the $\Sigma^r$ loci for $r=1$ and $r=2$.

\subsection{Euler characteristics of the $\Sigma^1$ locus}
%As an introductory example we will consider the simplest case of $\Si^i$-loci, the case of immersions.

 By definition, a map $f$ is an immersion, if $\Si^1(f)$ is empty. The simplest obstructions for the nonvanishing of $\Si^1$-points are provided by the Stiefel-Whitney classes of its normal bundle. Indeed, if $M^m\to\R^{m+l}$ is a codimension $l$ immersion, then it has a rank $l$ normal bundle $\nu$, and therefore $w_i(\nu)=0$ for $i>l$. In fact, the Stiefel-Whitney classes of the normal bundle describe the Thom polynomial: $[\Si^1(f)]=w_{l+1}(f)$.

The stable avoiding ideal $\mathcal{A}_{\Si^1(l)}$ is generated by the Stiefel-Whitney classes $w_i$ for $i\geq l$.  Consequently it does not give obstructions besides the Thom polynomials. However, studying the Stiefel-Whitney class of $\Si^1$ gives information on the geometry of the singular locus.

As a first application of the Parusi\'nski-Pragacz formula, here are the first few terms of the Segre-Stiefel-Whitney class of $\Si^1$:
\begin{proposition}\label{prop:Sigma1}
	The Segre-Stiefel-Whitney class of $\Si^1(l)$ is
	$$ \ssu_{\Si^1(l)}=\Phi^1+\Phi^3+\Phi^5+\ldots$$
where $\Phi^i$ is defined in \eqref{eq:phisl}, for example
\[
\Phi^1=\sum_{i=0}^{\infty}\sum_{j=0}^i\binom{l+i}{j}s_{l+1+j,1^{i-j}}.
\]
	Therefore, up to cohomological degree $3l+8$ ($\Phi^3$ starts at degree $3l+9$):
	\begin{equation}\label{eq:ssu1}
		(\ssu_{\Si^1(l)})_{\leq 3l+8}=\Phi^1_{\leq 3l+8}=\sum_{i=0}^{3l+8}\sum_{j=0}^i\binom{l+i}{j}s_{l+1+j,1^{i-j}}.
	\end{equation}
\end{proposition}

\begin{example}\label{ex:Sigma1}
As an application of Proposition \ref{prop:Sigma1}, we compute the Stiefel-Whitney class of the singular points of a generic map $f:\RP^{10}\to \R^{11}$.
This is an example where the fundamental class and the Euler characteristic of the singularity locus is 0, but the Stiefel-Whitney class is not, which implies that the Euler characteristic of a certain slice is nonzero.

 Writing $H^*(\RP^{10};\F_2)=\F_2[x]/(x^{11})$, we have (e.g.\ using Proposition \ref{prop:invSW}):
\[
\bar{w}(\RP^{10})=1+x+x^4+x^5
\]
%	$$w(\RP^{10})=(1+x)^{11}=1+x+x^2+x^3+x^8+x^9+x^{10}.$$
Proposition \ref{prop:Sigma1} describes $\ssu_{\Sigma^1(1)}$ up to degree 11, so completely for $\RP^{10}$:
\[
\ssu_{\Si^1(1)}=(s_2+s_{2,1}+s_{2,1,1}+\ldots)+(s_{3,1}+\ldots)+(s_4+\ldots)+\ldots
\]
Evaluating this on $\bar{w}(\RP^{10})$, by Theorem \ref{thm:ssw} \eqref{item:trans-pullback},
	\[ \ssu(\Si^1(f))=0+0+x^4+\ldots\]
In fact, by completing this computation, one obtains that there are no higher degree terms and that the whole expression is equal to $x^4$. We omit further details. To obtain the Stiefel-Whitney class we multiply with $w(\RP^{10})=1+x+x^2+x^3+x^8+x^9+x^{10}$:
	$$ \su(\Si^1(f))=x^4+x^5+x^6+x^7.$$
By Proposition  \ref{slice} we have:
		\[\su\big(\Si^1(f)\cap \RP^i\se \RP^i\big)=j^*\ssu\big(\Si^1(f)\big)\cdot w(\RP^i)=y^4\cdot (1+y)^{i+1}\]
		where $j:\RP^i\to \RP^{10}$ is a generic linear embedding and $H^*(\RP^i)=\F_2[y]/(y^{i+1})$. The Euler characteristic is then the coefficient of $y^i$: \[\chi\big(\Si^1(f)\cap \RP^i\big)= \binom{i+1}{i-4}.\] There are two cases when these binomial coefficients are odd: $i=4$ and $6$. This implies that for generic $f$ the intersection $\RP^4\cap \Si^1(f)$ is a smooth surface of odd Euler characteristic and $\RP^6\cap \Si^1(f)$ is 4-dimensional, of odd Euler characteristic. In particular, the Euler characteristic of $\Si^1(f)$ is even.

\end{example}

\subsection{Euler characteristics of the $\Sigma^2$ locus}
In this section, using Stiefel-Whitney classes and the Parusi\'nski-Pragacz formula, we will give a geometric meaning to the obstructions of Theorem \ref{k=t+1}.

\begin{theorem}\label{thm:euler-of-sigma2}
	Assume that the binary expansion of $n$ contains no consecutive 1's. Let $p$ be the largest number, such that $2^p$ divides $n$, and assume that $p\geq 1$. Then for a generic smooth map $f:\RP^n\to \R^{n+l}$, where $l=n/2-2^{p-1}-1$ the locus $\Si^2(f)$ is a smooth manifold of dimension $d=2^{p}-2$ of odd Euler characteristic. In particular it is nonempty, and for $p\geq 2$ unorientable.
\end{theorem}
Recall from Theorem \ref{k=t+1} that in these examples (case a) of Theorem \ref{k=t+1}) Thom polynomials do not obstruct:
\emph{all Thom polynomials $[\Si^2(g)]$ vanish for $g:\RP^n\to\R^{n+j}$ and $j\geq l$}, but the stable avoiding ideal gives an obstruction.

Theorem \ref{thm:euler-of-sigma2} gives a geometric interpretation of this obstruction result.
In fact, from Theorem \ref{thm:euler-of-sigma2} and Theorem \ref{k=t+1} it follows that the Euler characteristic is the only relevant obstruction in the stable avoiding ideal besides Thom polynomials for the existence of Morin maps $\RP^n\to \R^{n+l}$:
\begin{corollary}\label{thm:onlyeulerchar}
	Let $\bar{w}=1/w(\RP^n)$. Assume that $\tp_{\Sigma^2(j)}(\bar{w})=0$ for $j\geq l$ and $\chi(\Sigma^2(f))=0$ for a generic smooth map $f:\RP^n\to\R^{n+l}$.
	Then for all obstructions $s_\la$ in $\A_{\Sigma^2(l)}$, $s_\la(\bar{w})=0$.
\end{corollary}

 Again, we first illustrate the computation on the smallest example, cf.\ Example \ref{ex:Morinavoiding}.
\begin{proposition}\label{prop:20}
	Let $f:\RP^{20}\to \R^{27}$ be a generic map. Then
	$$\su(\overline{\Si}^2(f))=x^{20}$$
\end{proposition}
\begin{proof}
	By Theorem \ref{thm:sigmarsullivan}:
	$$\ssu(\overline{\Si}^2(f))=\Phi^2_{20,27}(f)$$
	since $\deg\Phi^i>20$ for $i>3$. Then
	$$ \Phi^2=\sum_{l(\mu)\leq 2}\sum_{l(\nu)\leq 2} D_{\mu,\nu}^{2,9}s_{9^2+\mu, \nu^T}=a_{18}+a_{19}+a_{20}.$$
	which is a sum of elements $a_i$ of cohomological degree $i$. Note that $a_{18}=\tp_{\Si^{2}(7)}(f)=0$.
	In the sum, there are 8 such partitions with degree $\leq 20$:
	$$ (9,9), (9,9,1), (10,9), (9,9,2), (9,9,1,1), (10,9,1), (10,10), (11,9)$$
	Recall that
	$$w(f)=\bar w(\RP^{20}) = 1+ x + x^2+ x^3+x^8+x^9+ x^{10} + x^{11}.$$
	As Schur polynomials, $(9,9),  (9,9,1), (10,9),(9,9,2), (10,10)$ all vanish after the substitution $\bar w(\PP^{20})$.
	The remaining three Schur polynomials $(9,9,1,1), (10,9,1)$ and $(11,9)$ are each equal to $x^{20}$. The coefficients $D_{\mu, \nu}^{2,9}$ are as follows:
	$$D_{(0,0),(2,0)}^{2,9}=\det\left(\begin{array}{cc}
		\binom{11}{1}&  \binom{8}{1}\\[3mm]
		\binom{10}{0}& \binom{7}{0}
	\end{array}\right)=3,\qquad
	D_{(1,0),(1,0)}^{2,9}=\det\left(\begin{array}{cc}
		\binom{11}{2}& \binom{9}{2} \\[3mm]
		\binom{9}{0}& \binom{7}{0}
	\end{array}\right)=19,$$
	$$	D_{(2,0),(0,0)}^{2,9}=\det\left(\begin{array}{cc}
		\binom{11}{3}&\binom{10}{3}  \\[3mm]
		\binom{8}{0}&\binom{7}{0}
	\end{array}\right)=45.$$
	Therefore $\ssu(\overline{\Si}^2(f))=x^{20}$. This implies that the Stiefel-Whitney class is
\[
		\su(\overline{\Si}^2(f))=\ssu(\overline{\Si}^2(f))\cdot w(\RP^{20})=x^{20}.
\]
\end{proof}
\begin{corollary} For a generic map $f:\RP^{20}\to \R^{27}$ the ${\Si}^2$ locus is a smooth surface of odd Euler characteristic. In particular it is nonempty and unorientable.
\end{corollary}

\begin{proof}[Proof of Theorem \ref{thm:euler-of-sigma2}:]
	We generalize the proof from the simplest case discussed in Proposition \ref{prop:20}.
	First, the dimension of the $\Si^2$-locus is
	\[d=n-2(l+2)=n-2(n/2-2^{p-1}+1)=2^p-2.\]
	%2^i\cdot(2^j+1-2^j)-2=2^i-2.\]
	\textbf{The case of $n\neq 2^p$.} Since $n\neq 2^p$, $n>2^{p+1}$.
	 Therefore, for $i>3$ \[\deg \Phi^i\geq i(i+l)\geq 4(n/2-2^{p-1})>n.\] 
	  Therefore all $\Phi^i$ evaluate to zero for $i>3$, so by Theorem \ref{thm:sigmarsullivan},
	\[
	\ssu\big(\clos{\Si^2}(\RP^n\to \R^{n+l})\big)=\Phi_{n, n+l}^2=\sum_{l(\mu)\leq 2} \, \sum_{l(\nu)\leq 2} D_{\mu,\nu}^{2,l+2}s_{(l+2,l+2)+\mu, \nu^T},
	\]
	where $D_{\mu,\nu}^{s,t}$ is defined in \eqref{eq:Dmunu}.
	Using Proposition \ref{prop:avoiding_obstructions}, %and that the Stiefel-Whitney class is universally supported on ${\Si}^2$ 
	we can reduce the proof to calculating the coefficients $D_{(p),(q-p)}^{2,l+2}$ %for $s=2$, $l=n/2-2^{p-1}-1$,
	%$\mu=(p)$ and $\nu=(q-p)$ 
	for $q=n-2(l+2)$ and $p=0,\dots,q$:
	\[ D_{(p),(q-p)}^{2,l+2}=\det\left(\begin{array}{cc}
		\binom{q+l+2}{p+1}&  \binom{p+l+1}{p+1}\\[3mm]
		\binom{q-p+l+1}{0}& \binom{l}{0}
	\end{array}\right).  \]
	Using Lucas's Theorem, \eqref{eqn:lucas} we obtain that $\binom{q+l+2}{p+1}\equiv 1$ and $\binom{p+l+1}{p+1}\equiv 0$, therefore $D_{(p),(q-p)}^{2,l+2}\equiv 1$. Since $q$ is even, and there are $q+1$ many such terms in
	\[\Phi^2|_{\bw}=\sum_{p=0}^{q} D_{(p),(q-p)}^{2,l+2}x^n,
	\]
	so  $\su(\overline{\Si}^2(f))=\ssu(\overline{\Si}^2(f))w(\RP^n)=x^n$. We also know that $\overline{\Si}^2(f)=\Si^2(f)$, therefore $\chi_2(\Si^2(f))=\int_{\PP^n}\su(\Si^2(f))=1$.

\textbf{The case of $n=2^p$.} In this case $l=-1$. Using the duality described in Remark \ref{rem:duality} we have $\Phi^i(-1)=\Phi^{i-1}(1)^\vee$. The contribution from $\Phi^i$ for $i>2$ is supported on $\Sigma^3(-1)$ and these classes evaluate to $0$ by Proposition \ref{prop:avoiding_obstructions}. It remains to count the non zero classes: they all have degree $n$ so the binomial coefficients in $\Phi^{2}(-1)=\Phi^{1}(1)^\vee$ are all of the form $\binom{n-1}{j}$ (see \eqref{eq:ssu1}), therefore equal to one, and according to Proposition \ref{prop:avoiding_obstructions} there are $n-1$ such terms, therefore the Segre-Stiefel-Whitney class evaluates to $x^n$. This again implies  that  the Stiefel-Whitney class evaluates to $x^n$, therefore for generic $f$ we have $\chi_2(\Si^2(f))=\int_{\PP^n}\su(\Si^2(f))=1$.

	For the final claim, notice that the dimension of the $\Si^2$-locus is $2^p-2$ which is of the form $4k+2$ for $p>1$ and the Euler characteristic of an oriented manifold of dimension $4k+2$ is even.
\end{proof}

\begin{remark}
	We deduce some geometric consequences of these results.
	We obtained that in the cases of Theorem \ref{thm:euler-of-sigma2}, that the mod 2 Euler characteristic $\chi(\Sigma^2(f))$ is non-zero. This implies that we calculated the unoriented bordism class of $\Sigma^2(f)$, whenever $\Sigma^2(f)$ is a surface, i.e.\ if $p=2$. Indeed, for a path-connected $X$ we have $N_2(X)\iso H_2(X;\Z_2)\oplus H_0(X;\Z_2)$, where the isomorphism is induced by the following: represent an element of $H_2(X;\Z_2)$ by an embedded smooth surface, and take its bordism class. Assuming $X$ is connected, the generator of $H_0(X;\Z_2)$ (if $X$ is connected), is mapped to the bordism class of $\P^2\to *$, i.e. $\P^2$ collapsed to a point of $X$. So the second component of the isomorphism is the mod 2 Euler class of our embedded smooth surface.
	
	Now we already showed that the fundamental class of $\Sigma^2(f)$ is zero, so we see that the information on the mod 2 Euler class determines the unoriented bordism class of $\Sigma^2(f)$.
\end{remark}

\section{Fold vs.\ Morin maps}\label{sec:fold_vs_morin}
In Section \ref{sec:Sigmai} we interpreted the obstructions to the existence of Morin maps as the Euler characteristic of the $\clos{\Si}^2$-points. In this section we give similar interpretations for the $\clos{A_2}$-locus. Contrary to the case of $\clos{\Si}^k$, there is currently no method available to compute the full Segre-SM or Segre-SW class of $\clos{A_2}(l)$. The difficulty is that we don't have a well understood stratification of the space of germs where $A_2$ is a stratum. 

On the other hand, the Stiefel-Whitney classes of the $A_k$ singularities can be calculated up to the avoiding ideal of $\Si^2$-points -- in particular, the formulas are valid for Morin maps. Based on these computations and further evidence, we formulate a conjectural relationship between cusp points and $\Si^2$-points, i.e.\ fold and Morin maps.
\subsection{A conjecture on fold and Morin maps}
The Segre-SW class of the $A_k$ singularities is described up to elements of the avoiding ideal of $\Si^2$-points by the following proposition:
\begin{proposition} \label{sw-ak4morin}
	The Segre-SW class of $\clos{A_k}(l)$ is equal to the following, up to elements of the avoiding ideal $\A_{\Si^2(l)}$:
	\[\ssu(\clos{A_k}(l))=\prod_{i=1}^{k}\prod_{j=1}^{l+1}\frac{b_j-ia}{1+b_j-ia},\]
	where $a,b_1\stb b_{l+1}$ denote the Stiefel-Whitney roots of $J^k(1,l+1)$.
\end{proposition}
\begin{proof} We prove the complex version. The real version will follow from Theorem \ref{thm:ssw} (\ref{item:bh}) and Theorem \ref{thm:mod2avoiding-sigma-i}.

	Consider the jet-space $J^k(1,l+1)=J^k(1)\otimes\C^{l+1}$. The closure of the $A_k(1,l+1)$-germs is a linear subspace:
	\[\clos{A_k}(1,l+1)=(x^{k+1})\otimes\C^{l+1}.\]
	The inclusion
	\[ J^k(1,l+1)\to J^k(n,l+n)\]
	is  Whitney transversal to $\clos{A_k}(n,l+n)$, therefore the SSM-class $s(\clos{A_k})$ for $\clos{A_{k}}\subset J^k(\infty,l+\infty)$ exists and can be written in quotient variables. (same arguments as in \cite{FeherRimanyi2012}). This implies that $s(\clos{A_{k}}(1,l+1))$ can be obtained from $s(\clos{A_{k}})$ by substitution of one source and $l+1$ target Chern roots. The kernel of this substitution is the avoiding ideal (in the ring of quotient Chern classes) of $\clos{\Sigma^2}(l)$ and elements of this avoiding ideal vanish for Morin maps.
\end{proof}
\begin{remark}
	To express this class in quotient Schur polynomials one can use the factorization theorem.
\end{remark}
Proposition \ref{sw-ak4morin} implies a further connection between the singularities $\clos{{A}_2}(l+1)$ and $\clos{\Sigma}^2(l)$:
\begin{theorem}\label{thm:sswa2(l+1)=sigma2lmod}
	The Segre-SW classes of $\clos{A}_2(l+1)$ and of $\clos{\Sigma}^2(l)$ are equal modulo the stable avoiding ideal of $\clos{\Sigma}^2(l+1)$.
\end{theorem}
\begin{proof}
	Similarly to Proposition \ref{sw-ak4morin} we have
	\[\ssu(\clos{\Sigma}^2(2,l+2))=\prod_{j=1}^{l+2}\frac{(b_j-a_1)(b_j-a_2)}{(1+b_j-a_1)(1+b_j-a_2)},\]
	where $a_1,a_2,b_1\stb b_{l+2}$ denote the Stiefel-Whitney roots of $J^k(2,l+2)$.
	This describes the stable class $\ssu(\clos{\Si}^2(l))$ up to the avoiding ideal $\A_{\clos{\Si}^3(l)}$. Setting $a_1\mapsto a$ and $a_2\mapsto 0$, we get the expression in Proposition \ref{sw-ak4morin} for $l+1$. This implies that the coefficients in the Schur expansion agree for all partitions $\la$  which do not contain the $2\times (l+3)$ rectangle (the partitions generating $\A_{\clos{\Si}^2(l+1)}\se \A_{\clos{\Si}^3(l)}$).
\end{proof}

We also found that the Segre-SW classes of $\clos{A_2}(l+1)$ and $\clos{\Si}^2(l)$ are equal in the range where we can compute Segre-SW classes of $\clos{A_2}(l+1)$ via the restriction equation method described below in Section \ref{subsec:restriction}.
(This is only a real phenomenon: in the complex case already the Thom polynomials of $A_2(l+1)$ and $\Si^2(l)$ are not equal.) For example, using calculations of Rich\'ard Rim\'anyi of Segre-SM Thom polynomials we have that the Segre-SW classes of $\clos{A_2}(l+1)$ and $\clos{\Si}^2(l)$ are equal up to degree 8 for $l=0$ and up to degree 12 for $l=1$. Based on this, we propose the following conjecture.
\begin{conjecture}\label{conj:sw-a2-si2} The Segre-SW classes of $\clos{A_2}(l+1)$ and $\clos{\Si}^2(l)$ are equal. \end{conjecture}

\begin{remark}\label{rem:a2-si2} 
	Recall that in Corollary \ref{cor:a2-si2} we showed that the stable avoiding ideals of $\clos{A_2}(l+1)$ and $\clos{\Si}^2(l)$ are equal. This means that if there is an obstruction $\xi\in \A_{A_2(l+1)}$, which does not vanish for some $f:M\to N\times \R$, ---implying that $\clos{A_2}(f)$ is nonempty--- then the same obstruction shows that $\clos{\Si}^2(\pi\circ f:M\to N)$ is nonempty. It seems that there is a much stronger relationship:
	We conjecture that these are cobordant classes under an appropriate resolution.\end{remark}

We can formulate obstruction theoretic versions, for example:
\begin{conjecture}\label{conj:fold_Morin} Let $f:M\to N$ be a Morin map. Then there is a lift $\tilde{f}:M\to N\times\R$ of $f$, which is a fold map. \end{conjecture}
A similar statement for cusp maps was proved in \cite{CST}. 
\subsection{Euler characteristics of the $A_2$-locus}\label{subsec:A2eulerchar}
We make some observations regarding the Euler characteristic of the $\clos{A}_2$-locus of maps of $\RP^n$ for the case a) described in Theorem \ref{k=t+1}.

Saeki and Sakuma \cite[Theorem 4.1(1)]{SaekiSakuma1999} proved that the Euler characteristic of the $\clos{A_2}$-locus of a Morin map $f:M^4\to \R^4$ is equal to the Euler characteristic of $M$:\begin{equation}\label{eq:A2eulerchar}
	\chi(\clos{A_2}(f))=\chi(M)=\int_M w_4(M).
\end{equation}  
For instance, this implies that if $f:\RP^4\to \R^4$ is a Morin map, then the Euler characteristic of the cusp-locus $\chi(\clos{A_2}(f))$ is odd. We can generalize this example using Theorem \ref{thm:sswa2(l+1)=sigma2lmod} and Theorem \ref{thm:euler-of-sigma2}:

\begin{theorem}\label{thm:A2eulerchar}
	Assume that the binary expansion of $n$ contains no consecutive 1's. Let $p$ be the largest number, such that $2^p$ divides $n$, and assume that $p\geq 1$. Then for a generic smooth, Morin map $f:\RP^n\to \R^{n+l+1}$, where $l=n/2-2^{p-1}-1$, the locus $\clos{A_2}(f)$ is a smooth manifold of dimension $d=2^{p}-2$ of odd Euler characteristic. In particular it is nonempty and for $p\geq 2$, unorientable.
\end{theorem}

We remark that in Theorem \ref{thm:A2eulerchar}, the existence of a Morin map has to be decided on a case-by-case basis. The theorem holds without the Morin condition in the range where we could compute the Segre-SW class (up to degree 9, so for $\RP^4$ and $\RP^8$). For instance:
\begin{example}
	Any smooth 4-manifold $M$ with non-zero signature does not have a Morin map $f:M\to \R^4$, since its integer Thom polynomial (e.g.\ \cite{Ronga}) is nonzero: \[\int_M[\Si^2(f)]=\int_M p_1(M)=3\si(M).\] On the other hand, Saeki and Sakuma's theorem \eqref{eq:A2eulerchar} also holds for smooth generic maps $f$ (by \eqref{eq:A2closedM4R4} below), even though the $\clos{A_2}$-locus is no longer smooth. So for a smooth, generic ---not necessarily Morin--- map $f:\RP^4\to \R^4$, $\chi(\clos{A_2}(f))$ is odd.
	
	In fact, it turns out that the the Thom polynomial of $A_2$ and $\Si^2$-points ($w_2$ and $p_1$), and the Euler characteristic of the $A_2$-points ($w_4$) are the only obstructions for having a fold map $f:M^4\to\R^4$, see \cite[Corollary 3.4]{SadykovSaekiSakuma2010} for the precise statement. 
\end{example}
If Conjecture \ref{conj:sw-a2-si2} holds, the Morin condition in Theorem \ref{thm:A2eulerchar} can be dropped altogether. Then the $\clos{A_2}$-locus will not necessarily be smooth, but the conclusion about odd Euler characteristic still holds.

\section{Calculating Stiefel-Whitney classes and Euler characteristics of other singularity loci}\label{sec:singularity}
In this section we compute the Stiefel-Whitney classes of singularity loci other than the $\Si^r$-classes and $A_r$-classes. For the relevant notions from singularity theory and for a list of prototypes of singularities we refer to \cite{AVGL}, \cite{tpp} and \cite{primer}. 
%The applications described in Section \ref{sec:applications} are based on these results.
\subsection{Restriction equations}\label{subsec:restriction}
The method of restriction equations was developed in \cite{Rimanyi2001}. Using the \quot{smallness} axiom (III) of \cite[Theorem 2.7]{FeherRimanyi2018} Rim\'anyi expanded his restriction equation method to calculate CSM and SSM classes of singularities. His paper \cite{csmtp} is in preparation, however the results are already available on \cite{tpp}.
Hoping that his paper will be published soon, we only demonstrate the method on a small example: we calculate $s(A_2):=\ssm(A_2)$ for $l=0$ up to degree 3 using restriction equations. This was first calculated by Ohmoto in \cite[4.3]{Ohmoto2016} by a slightly different method.

$s(A_2)$ is a formal series: $s(A_2)=s_2(A_2)+s_3(A_2)+\cdots$, where $s_2(A_2)$ is the Thom polynomial of $A_2$. It is well known that $[A_2]=s_2(A_2)=c_1^2+c_2$.

Let us write the unknown $ s_3(A_2)$ in the general form $Ac_1^3+Bc_1 c_2+Cc_3$. In \cite{Rimanyi2001} we find that the restriction homomorphisms are given by
\[c|_{A_i}=1+ia-ia^2+ia^3-\cdots.\]
Since $ s_3(A_2)$ is supported on $A_2$, we have
\[  (Ac_1^3+Bc_1 c_2+Cc_3)|_{A_1}=0,\]
which implies that
\begin{equation}\label{eq:1}
	A-B+C=0.
\end{equation}
%\[   \]
Recall from \cite{Rimanyi2001} that the source weights are $\mathcal{W}(S_{A_2})=\{a,2a\}$ and $\mathcal{W}(S_{A_3})=\{a,2a,3a\}$. Therefore the restriction to $A_2$ is
\[  \left.\frac{a\cdot2a}{(1+a)(1+2a)}\right|_3=a\cdot2a(1-3a+\cdots)|_3=-6a^3,\]
implying that
\begin{equation}\label{eq:2}
	8A-4B+2C=-6.
\end{equation}

To find a third equation we also restrict to $A_3$. The smallness axiom implies that $\csm_3(A_2)|_{A_3}=0$, therefore
\[ \csm(A_2)|_{A_3}=[A_2]|_{A_3}=(c_1^2+c_2)|_{A_3}=9a^2-3a^2=6a^2,\]
therefore
\[ s_3(A_2)|_{A_3}=\left.\frac{\csm(A_2\subset S_{A_3})}{c(S_{A_3})}\right|_3=\left.\frac{6a^2}{(1+a)(1+2a)(1+3a)}\right|_3=-36a^3.\]
This implies that
\begin{equation}\label{eq:3}
	27A-9B+3C=-36.
\end{equation}
Solving the system of equations \eqref{eq:1}, \eqref{eq:2}, \eqref{eq:3}, we obtain the $\ssm$ class of $A_2$ up to degree 3:
\[
\ssm(A_2)=(c_1^2+c_2)-(3c_1^3+6c_1c_2+3c_3)+\ldots.
\]
\begin{remark}
	Notice that with mod 2 coefficients the method of restriction equations is typically not solvable, e.g.\ \eqref{eq:2} gives the equation $0=0$ modulo 2. However solving the restriction equations for complex contact singularities and applying the Borel-Haefliger type Theorem \ref{thm:ssw}, \eqref{item:bh} we get the corresponding real result.
\end{remark}
\subsection{Stiefel-Whitney classes of singularities}\label{sec:swclasses}

\begin{definition}
  Let $\eta$ be a contact singularity of relative codimension $l$. Then the formal power series
\begin{equation}\label{eq:sweta}
\su_\eta:=\ssu_{\eta}\cdot(1+t_1+t_2+\cdots)\in\F_2[[\www]][[\ttt]]
\end{equation}
is called the \emph{Stiefel-Whitney class of $\eta$}.
\end{definition}
The universal property Theorem \ref{thm:ssw} and the transversality result Theorem \ref{thm:generic-pullback} implies that for a generic map $f:M\to N$ and a singularity $\eta$ we can express the Stiefel-Whitney class of the $\eta$-locus of $f$ in terms of the normal Stiefel-Whitney classes $w(f)$ and the Stiefel-Whitney classes of the source $w(M)$:
\begin{proposition}\label{prop:sw}
	Let $\eta$ be a contact singularity of relative codimension $l$. Then  $\su_{\eta}$
%\[\su_{\eta}:=\ssu_{\eta}\cdot(1+t_1+t_2+\cdots)\in\Z[[w_1,\dots,w_m,\dots]][[t_1,\dots,t_m,\dots]] \]
has the universal property that for any generic map $f:M^n\to N^{n+l}$ with $M$ compact of codimension $l$
\begin{equation}\label{eq:su_eta_f}
	 \su(\eta(f))=\su_\eta\big(w_i\mapsto w_i(f), t_i\mapsto w_i(M)\big),
\end{equation}
for $i\leq n$, i.e.\ $w_i\mapsto 0$ and $t_i\mapsto 0$ for $i>n$.
\end{proposition}
If $N=\R^{n+l}$ or more generally if $f^*TN$ is trivial, then $\su_{\eta}$ can be expressed solely in terms of $t_i$, the Stiefel-Whitney classes of the source. Using that $w(f)=1/w(M)$, and setting
\begin{equation}\label{eq:swhateta}
	\hat\su_{\eta}:=\su_{\eta}(w_i\mapsto \bar t_i)\in \F_2[[\ttt]],
\end{equation}
where $\bar t_i$ are the formal inverse classes:
\[ 1+\bar t_1+ \bar t_2+\cdots:=\frac{1}{1+t_1+t_2+\cdots},\]
we obtain
\[
 \su(\eta(f))=\hat\su_\eta\big(t_i\mapsto w_i(M);i=1,\dots,n\big).
\]
For some explicit computations of $\su_\eta$ and $\hat{\su}_\eta$, see Example \ref{ex:char_series} below as well as Appendix \ref{app:charseries}.
\subsection{Euler characteristics of singularity loci}\label{sec:eulerchar}
Using the relationship \eqref{eq:su} between the Stiefel-Whitney class and the Euler characteristic, we obtain that if $\eta$ is a contact singularity of relative codimension $l$, then
\begin{equation}\label{eq:what-euler}
  \chi(\eta(f))=\int\limits_M\hat\su_\eta\big(t_i\mapsto w_i(M)\big)
\end{equation}
for $i\leq n$, and any generic map $f:M^n\to \R^{n+l}$ with $M$ compact of codimension $l$. This equation shows that the Euler characteristic of the $\eta$-locus can be expressed in terms of Stiefel-Whitney numbers, implying that it is invariant under cobordism. To emphasize this, and to make calculations easier we introduce the \emph{characteristic series} $\chi_\eta: \mathcal{N}_*\to \Z_2$, where $\mathcal{N}_*$ denotes the unoriented cobordism ring. Since Stiefel-Whitney numbers determine the unoriented cobordism class of a manifold, we have a surjective map
\[ \F_2[\tatata]\to \Hom(\mathcal{N}_*,\Z_2),\]
where a monomial $\tau^I$ in the variables $\tau_i$ represents the corresponding Stiefel-Whitney number $w_I[M]$ of $M$. For example
\[\chi_{\bar A_0(l)}=1+\tau_1+\tau_2+\tau_3+\tau_4+\tau_5+\tau_6+\cdots=1+\tau_2+\tau_4+\tau_6+\cdots \]

Notice that these expressions for the characteristic series are not unique: they are only determined up to the universal relations between the Stiefel-Whitney numbers.
\begin{example} \label{ex:char_series} For the (open) singularity locus $A_2(0)$ we can consult Rim\'anyi's table \cite{tpp}. Taking the SSM class modulo 2 by Theorem \ref{thm:ssw} we obtain the Segre-SW class:
\begin{equation}\label{eq:A2ssw}
	\begin{split}
		\ssu_{A_2(0)}=&\,(w_2+w_1^2)+(w_1^3+w_3)+(w_2^2+w_4)+(w_2^2w_1+w_5)+\\
		&\,(w_2^3+w_6+w_1^6+w_4w_2+w_3w_1^3+w_2w_1^4+w_3w_2w_1+w_4w_1^2)+\ldots
	\end{split}
\end{equation}
and after applying the transformations \eqref{eq:sweta} and \eqref{eq:swhateta}, we obtain
\begin{equation}\label{eq:A2su}
	\begin{split}
		\hat\su_{A_2(0)}=&\,t_2+(t_3+t_2t_1)+(t_4+t_1^2t_2+t_3t_1+t_2^2)+(t_5+t_4t_1+t_3t_1^2+t_2t_1^3)\\
		&(t_2t_1^4+t_2^2t_1^2+t_2^3 +t_3t_1^3+t_3t_2t_1+t_4t_1^2+t_4t_2+t_5t_1+t_6)+\ldots
	\end{split}
\end{equation}
This implies that
\begin{equation}\label{eq:A2chi}
	\begin{split}
		\chi_{A_2(0)}=&\,\tau_2+(\tau_3+\tau_2\tau_1)+(\tau_4+\tau_1^2\tau_2+\tau_3\tau_1+\tau_2^2)+(\tau_5+\tau_4\tau_1+\tau_3\tau_1^2+\tau_2\tau_1^3)\\
		&(\tau_2\tau_1^4+\tau_2^2\tau_1^2+\tau_2^3 +\tau_3\tau_1^3+\tau_3\tau_2\tau_1+\tau_4\tau_1^2+\tau_4\tau_2+\tau_5\tau_1+\tau_6)+\ldots
	\end{split}
\end{equation}

Using the Dold relations between characteristic numbers (see Appendix \ref{app:charseries}) we can simplify the expression:
\begin{equation}\label{eq:chia2}
\chi_{A_2(0)}=\,\tau_2+0+(\tau_4+\tau_1^2\tau_2+\tau_2^2)+0+(\tau_2\tau_4)+\ldots
\end{equation}
For example, this means that the Euler characteristic of the (open) cusp locus $A_2(M^6\to \R^6)$ is equal to the Stiefel-Whitney number $\int\limits_M w_2w_4$.
\end{example}

For similar results for the closed singularity loci, we can sum the characteristic series 
%of $A_2(0),A_3(0),A_4(0)$ and $I_{2,2}(0)$ 
computed in \eqref{eq:chi_singularity}: 
\begin{theorem}\label{thm:closure-cusp} The characteristic series for the following singularities of equidimensional maps is
\begin{equation}\label{eq:A2closedM4R4}
		\begin{split}
 \chi_{\clos{A_2}(0)}&=\tau_2+0+\tau_4+0+(\tau_{4}\tau_2+\tau_3\tau_2\tau_1+\tau_6)+0+(\tau_4\tau_2^2 + \tau_6\tau_2 + \tau_4^2 + \tau_8)+\ldots\\
	\chi_{\clos{A_3}(0)}&=0+\tau_2^2+0+(\tau_3\tau_2\tau_1+\tau_6)+0+(\tau_2^4+\tau_4\tau_2^2+\tau_5\tau_3)+\ldots\\
%\end{equation}
%\begin{equation}
	\chi_{\clos{A_4}(0)}&= 0 + 0 + \tau_3^2 + 0 + (\tau_6\tau_2+\tau_4\tau_2^2)+\ldots\\
%	 4: h[3, 1],
%	5: h[3, 1, 1],
%	6: h[3, 1, 1, 1] + h[3, 3] + h[4, 1, 1] + h[4, 2] + h[5, 1],(c+b+a+c+a)=b
%	7: h[3, 1, 1, 1, 1] + h[3, 3, 1] + h[4, 2, 1] + h[4, 3] + h[5, 2],
%	8: h[3, 1, 1, 1, 1, 1] + h[3, 3, 2] + h[4, 1, 1, 1, 1] + h[4, 2, 1, 1] + h[4, 2, 2] + h[7, 1]}
%=[6,2] + [4,2,2]
%\end{equation}
%\begin{equation}
	\chi_{\clos{I_{2,2}(0)}}&=\tau_2^2+0+(\tau_6+\tau_4\tau_2)+0+(\tau_2^2\tau_1^4+\tau_2^3\tau_1^2+\tau_2^4+\tau_6\tau_2)+\ldots \\
	 %4: h[2, 2] + h[3, 1],
%	5: h[2, 2, 1] + h[3, 1, 1],
%	6: h[2, 2, 1, 1] + h[2, 2, 2] + h[3, 1, 1, 1] + h[3, 3] + h[4, 1, 1] + h[5, 1]=a+b+c+b+a+a=a+c
%	7: h[2, 2, 1, 1, 1] + h[3, 1, 1, 1, 1] + h[3, 2, 2] + h[3, 3, 1],
%	8: h[2, 2, 1, 1, 1, 1] + h[2, 2, 2, 1, 1] + h[2, 2, 2, 2] + h[3, 1, 1, 1, 1, 1] + h[3, 2, 2, 1] + h[3, 3, 1, 1] + h[3, 3, 2] + h[4, 1, 1, 1, 1] + h[4, 2, 1, 1] + h[4, 2, 2] + h[4, 3, 1] + h[5, 2, 1] + h[5, 3] + h[7, 1]=
%h[2, 2, 1, 1, 1, 1] + h[2, 2, 2, 1, 1] + h[2, 2, 2, 2]   + h[6, 2]  
%	\end{equation}
%\begin{equation}
	\chi_{\overline{I_{2,3}}(0)}&=0+(\tau_6+\tau_4\tau_2)+0+(\tau_5\tau_3+\tau_4\tau_2^2)\\
	% h[3, 2, 2, 1] +  h[4, 2, 2]
%\end{equation}
	\chi_{\clos{I_{2,4}(0)}}&=(\tau_3^2+\tau_4\tau_2)+0+\tau_4^2\ldots \\
	%h[3, 2, 2, 1] + h[3, 3, 1, 1] + h[3, 3, 2] + h[4, 2, 2] + h[4, 3, 1] + h[4, 4] + h[5, 2, 1] + h[5, 3]=
	% h[4, 4]
%\begin{equation}
	\chi_{\overline{I_{2,5}}(0)}&=\tau_5\tau_2+0+\ldots
\end{split}
\end{equation}
\end{theorem}
Note that each term in these expressions describes a geometric problem. For instance, Saeki and Sakuma's theorem (discussed in Section \ref{subsec:A2eulerchar}) about the Euler characteristic of the $\clos{A_2}$-locus of maps $M^4\to\R^4$ corresponds to the term $\tau_4$ in $\chi_{\clos{A_2}(0)}$.
\begin{remark}
	One can extend these formulas to maps $f:M\to N$ with non-parallelizable target $N$. To do so, one can use the formulas \eqref{eq:ssu_singularity} describing the Segre classes of the singularity locus, and integrate the corresponding Stiefel-Whitney classes given by the relation \eqref{eq:ssu_su}. The resulting formulas are more complicated, as they now involve two sets of variables $w_i$ and $t_i$.
\end{remark}
We give a sample computation of Segre classes of more complicated singularity loci.
\begin{example}\label{ex:Grassmannian}
	We compute the Segre classes of some singularity loci $\eta$ of Thom-Boardman type $\Si^2$. In particular, these computations imply that there is no Morin map $\Gr_2(\R^6)\to\R^8$. The $\ssm$ classes of singularity loci in codimension 0 were computed by Rim\'anyi \cite{tpp} up to degree 8. By Theorem \ref{thm:ssw} \eqref{item:bh}, the $\ssu$ classes are the same mod 2, so the next step is to compute $w(\Gr_2(\R^6))$. The Stiefel-Whitney classes of $\Gr_2(6)$ can be computed using standard methods of Schubert calculus:
	\[
	w(\Gr_2(\R^6))=1+\si_{2,2}+\si_{3,1}+\si_{3,3}+\si_4+\si_{4,2}+\si_{4,4}
	\]
	where $\si_\la$ denotes the mod 2 fundamental class of the Schubert variety (cf.\ \cite{BorelHaefliger1961}).
	Substituting this into the respective formulas \cite{tpp} or \eqref{eq:ssu_singularity} (caution, the Chern classes $c_i$ on \cite{tpp} denote inverse Chern classes), we obtain the $\ssu$ classes of the (open) singularity loci:
	\[
	\ssu_{I_{22}}=0,\qquad 	\ssu_{I_{23}}=\si_{4,4},\qquad \ssu_{I_{33}}=0,\]
	\[\ssu_{I_{25}}= 0,\qquad \ssu_{I_{34}}= \si_{4,4},\qquad \ssu_{(x^2,y^3)}=\si_{4,4},\]
	\[\ssu_{I_{26}}= \si_{4,4},\qquad
	\ssu_{I_{35}}= 0,\qquad
	\ssu_{I_{44}}= \si_{4,4},\qquad  \ssu_{(x^2+y^3,xy^2)}=\si_{4,4}.
	\]
	There are some interesting phenomena to note. Since $\su_{\eta}(f)=w(\Gr_2(6))\cdot \ssu_{\eta}(f)$, and since for all of these above singularities, the Segre-SW class is either $\si_{4,4}$ or 0, we computed the mod 2 Euler characteristics of these (open) singularity loci; they are 0 whenever $\ssu=0$, and 1 otherwise.
	
	Second, notice that the sum of the $\ssu$ classes of these singularity loci is 0, i.e.\ $\ssu_{\Si^2(0)}(f)=0$. This can also be directly obtained from the formulas of Parusi\'nski-Pragacz. Since the $\ssu$ class of $\Si^2(0)$ vanishes, it provides no obstruction to the existence of a Morin map of codimension 0. On the other hand, the Thom polynomials of the singularity loci (e.g. $\tp_{I_{44}}= \si_{4,4}$ etc.) do provide such an obstruction.
	
\end{example}
\section{Further examples and observations}\label{sec:hierarchy}
In this section, we consider some further results that can be obtained using the techniques of this paper. Most of these examples are motivated by the following question: what is the minimal set of geometric elements from $\A_\eta$ that one has to compute, in order to conclude that $\A_\eta$ vanishes? We also consider the relationship of $\ssu$ with Steenrod operations.
\subsection{Geometric obstructions in the avoiding ideal}

In some (rare) cases, the Thom polynomial of $\eta$ is the only obstruction for the existence of an $\eta$-map $f:M\to \R^n$. For instance, a manifold of dimension $4k+2$ admits a fold map into $\R^4$ if and only if its Thom polynomial $\tp_{A_2(4k-2)}=t_{4k}$ vanishes by a theorem of Sadykov, Saeki and Sakuma \cite[Theorem 4.6]{SadykovSaekiSakuma2010}.

However, this is not the case in general: the following example illustrates that Stiefel-Whitney classes occasionally provide sharper estimates than the Thom polynomials of more complicated singularities (cf.\ Proposition \ref{prop:avoiding_complicated}).
%(Computing all Thom polynomials is not available in the case of Lee's examples, since there are arbitrarily high dimensions.)
\begin{example}\label{ex:rp4rp6}
	We show that any map $f:\RP^4\times \RP^6\to \R^{11}$ has $A_2$ points. We have
	\[w(\RP^4\times \RP^6)=(1+x+x^4)(1+y+y^2+y^3+y^4+y^5+y^6)\]
	where $H^*(\RP^4\times \RP^6)=\F_2[x,y]/(x^5,y^7)$. Its inverse is
	\[\bar w = 1 + (x+y)+(x^2+xy)+(x^3+x^2y)+x^3y\]
	Substituting these classes into $\ssu_{A_2(1)}$ (see \cite{tpp}) the Segre-SW class of the $A_2$-locus is
	\[\ssu_{A_2(1)}(f)=x^4y^4.\]
	In fact, the Thom polynomials of the other singularities in $\A_{A_2(1)}$ vanish. In degrees $d\leq 10$, these are (for their Thom polynomials, see \cite{tpp} and the references therein):
%	\begin{equation}\label{eq:singlist}
	\[
	A_2(1), A_3(1), III_{2,2}(1), A_4(1), III_{2,3}(1), A_2(2), A_2(3), A_2(4), A_3(2),I_{2,2}(2).
	\]
%	\end{equation}
	In conclusion, $(\ssu_{A_2(1)})_{+4}\in \A_{A_2(1)}$ is an obstruction not generated by the Thom polynomials of more complicated singularities.
%	Using Terpai's result Theorem \ref{thm:Terpai}, we can calculate all obstructions coming from the avoiding ideal of $A_2(1)$. The smallest non-zero obstruction is
%	\[\bar w_3\bar w_2 +\bar w_4 \bar w_1 = (x^2+xy)(x^3+x^2y)+x^3y(x+y)=x^4y.\]	
\end{example}

The following example illustrates that sometimes even though $\ssu_{\eta}(f)=0$, $\A_\eta(f)$ can be nonzero.
\begin{example}\label{ex:DoldWu}
	Let $M=\operatorname{SU}(3)/\operatorname{SO}(3)$, which is a 5-dimensional simply connected manifold also known as the Wu manifold \cite{Barden}. We show that $M$ has no Morin map to $\R^6$. Its cohomology ring is
	\[H^*(M;\F_2)=\Lambda[w_2(M),w_3(M)],\]
	see \cite[Lemma 1.1, 1.2]{Barden}, \cite[Theorem 6.7.2]{MimuraToda}, so its inverse Stiefel-Whitney class is $\bar{w}(M)=1+w_2+w_3$. Let $f:M\to \R^6$ be a generic smooth map. The Thom polynomial $[A_2(1)](f)=\bar w_2^2+\bar w_1\bar w_3\in H^4(M)$ vanishes, since $M$ is simply connected.  The $A_2(1)$-locus is smooth: its singular points are the $\operatorname{III}_{2,2}(1)$-points, which are 6 codimensional, and therefore empty. So the set of $A_2(1)$ points is either empty, or a smooth 1-dimensional submanifold, so its Euler characteristic also vanishes. As we saw in the previous example the avoiding ideal of $A_2(1)$ also contains the element $w_2w_3+w_1w_4$ which has nonzero value on $f$. Therefore any map of the Wu manifold to $\R^6$ has $A_2$-points. As we have remarked in Corollary \ref{cor:a2-si2}, $\A_{A_2(1)}=\A_{\Si^2(0)}$, so there is a corresponding result on Morin maps: any map of the Wu manifold to $\R^5$ has $\Si^2$ points.
	
It is possible to identify the obstruction $w_2w_3+w_1w_4$ as the pushforward of $w_1(\ker df)$, where $\ker df$ is the line bundle of the kernels of the derivative of $f$ over the $A_2(1)$-locus. This calculation will be published elsewhere.
\end{example}

In the following example, we illustrate the simple fact that since the Euler characteristic $\chi(\eta(f))$ is equal to a Stiefel-Whitney number of $M$, it is an invariant of the cobordism class $[M]$ (cf.\ Section \ref{sec:eulerchar}), however the Thom polynomial $\tp_\eta(f)$ is not.

\begin{example}\label{ex:Dold5}
	Consider maps $f:\RP^5\to \R^6$. Since $\bar{w}(\RP^5)=1+x^2$, the Thom polynomial is equal to $\tp_{\Si^1}(f)=x^2$ which is nonzero. Since $\RP^5$ and $(S^1)^{\times 5}$ are both cobordant to zero, and $\chi(\Si^1(f))$ is a Stiefel-Whitney number, $\chi(\Si^1(f))=0$ for both spaces.
	
	In fact, one can show using the Smale-Hirsch theorem that $(S^1)^{\times 5}$ can be immersed to $\R^6$.
	%mely kobordizmusosztalyok reprezentalhatok $\eta$-elkerulo lekepezessel $\R^{n+l}$-be?
\end{example}

%even though the avoiding ideal $\A_\eta$ does not provide any obstructions for $f:M\to \R^n$, there is still no $\eta$-avoiding map homotopic to $f$.
Finally, we remark that there exist more complicated obstructions than the avoiding ideal. For instance, for $\RP^{N}\to \R^{N+l}$, $N=2^n-1$ the avoiding ideal is zero $\A_\eta=(0)$ for any singularity $\eta$, since $w(\RP^N)=1$. On the other hand, $\RP^{15}$ does not immerse to $\R^{21}$, nor does $\RP^{31}$ immerse to $\R^{52}$ by \cite{James1963}. See \cite{Davis} for an extensive overview on immersion results of $\RP^n$ to Euclidean space.

\subsection{Stiefel-Whitney classes and Steenrod operations}

	Some Segre-SW classes of $\eta$ always vanish whenever their Thom polynomial vanishes, so they do not provide additional obstructions.  %We give an alternative proof of a special case of Goldstein's theorem:
\begin{proposition}\label{prop:sq1}
		If $Z\se X$ be a closed subvariety with singular set $S$ of codimension $\geq 2$, then $\ssu_{+1}(Z)=\Sq^1[Z]$.
\end{proposition}
\begin{proof}
Let $f:Y\to X$ be a resolution of $Z$. Then $f_!w(Y)=\su(Z_1)$, where %for any constructible subset $S\se X$ let
		$$Z_1:=\{z\in Z:\chi_2(f^{-1}(z))=1\}.$$
		%	$$S_1:=\{s\in S:\chi_2(f^{-1}(s))=1\}.$$
Since $f$ is a resolution, the smooth part $U\se Z$ is contained in $Z_1$. The singular set decomposes as $S=S_0\amalg S_1$ and
      $$f_!w(Y)=\su(U)+\su(S_1).$$
If the singular set $S$ has codimension $\geq 2$ in $Z$, then
		$$ \su_{+1}(Z)=\su_{+1}(U)=f_!w_1(Y)=\Sq^1[Z]+w_1(X)\cdot [Z],$$
where the last statement follows from Atiyah and Hirzebruch's theorem \cite[Satz 3.2]{AtiyahHirzebruch} (a Grothendieck-Riemann-Roch type theorem for $\Sq$): $f_!(1/w(Y))\cdot w(X)=\Sq[Z]$. The Segre-Stiefel-Whitney class is the degree $+1$ part of
		$$ \ssu_{+1}(Z)=\left[([Z]+\Sq^1[Z]+w_1(X)\cdot [Z]+\ldots)(1+w_1(X)+\ldots)\right]_{+1}=\Sq^1[Z]$$
\end{proof}
	For instance, the Schubert varieties of a flag variety have singular sets of codimension $\geq 2$, so they all satisfy the condition of the Proposition. The same holds for the $\clos{\Si}^i$-loci. The conclusion is not always satisfied:
	\begin{example}
		%	Even though $\ssu(A_4(0))=s_{1,1,1,1}+s_{2,1,1}+2s_{1,1,1,1,1}+2s_{2,2,1}+2s_{3,1,1}\ldots$ by \cite{tpp},  $\Sq^1(s_{1,1,1}+s_{2,1})=$. Indeed,
		%$\ssu(I_{2,2}(0))=s_{2,2}+$
		By \cite{tpp} one has $$\ssu(\clos{A_3(0)})=s_{1,1,1}+s_{2,1}+s_{1,1,1,1}+s_{2,1,1}+s_{2,2}+\ldots,$$
		however $\Sq^1(s_{1,1,1}+s_{2,1})=s_{1,1,1,1}+s_{2,1,1}$. Indeed, the singular set of $\clos{A_3(0)}$ is $\clos{I_{2,2}(0)}$ which is 1 codimensional.

In fact, this calculation shows that $\clos{A_3(0)}$ is not smooth at the $I_{2,2}(0)\cup II_{2,2}(0)$ points. This is well known as the geometry is classically well understood. But it shows how to use $\ssu-\Sq^1$ as an obstruction for smoothness in codimension 1.
\end{example}
	
In some particular $i>1$ cases $(\ssu_{\eta})_{+i}$ vanishes, whenever $\tp_\eta$ vanishes.
	\begin{example}\label{ex:hierarchy}
		Whenever $\tp_{A_2(1)}=s_{2,2}$ vanishes, $(\ssu_{A_2(1)})_{+2}$ also vanishes. Indeed,
		\[\Sq^2(s_{2,2})=s_{2,2,1,1} + s_{2,2,2} + s_{3,2,1} + s_{3,3} + s_{4,2}\]
		which can be computed using methods of \cite{Lenart1998} or by applying the splitting principle to $H^*(BGL(n,\R))$. On the other hand, $(\ssu_{A_2(1)})_{+2}$ is equal by \cite{tpp} to
		\[
		(\ssu_{A_2(1)})_{+2} = s_{3,3}+s_{2,2,1,1} = \Sq^2(\tp_{A_2(1)}) + s_2\cdot \tp_{A_2(1)}
		\]
		So the vanishing of the Thom polynomial implies vanishing of $\ssu_{+2}$. This is not true for the higher $\ssu$ classes of $A_2(1)$, e.g.\ see Example \ref{ex:rp4rp6} with $\ssu_{+4}(A_2(1))$.
	\end{example}

\appendix

\section{Computations: Stiefel-Whitney and characteristic series}
\label{app:charseries}

We summarize computations of  Segre-Stiefel-Whitney classes $\ssu_\eta$ and characteristic series $\chi_\eta$ of some singularities $\eta$. For $l=0$, the restriction equations only work up to degree 8 (a moduli of contact orbits appears in codimension 9 singularities). To keep the formulas relatively compact, we only write out the computations up to degree 6. 

First, using Theorem \ref{thm:PP} of Parusi\'nski-Pragacz and Theorem \ref{thm:ssw} (5), we write the Segre-SW class of the first two (open) $\Si^i$'s in codimension 0, up to degree 6:
\begin{equation}
		\begin{split}
	\ssu_{\Si^1(0)}=&\, w_1+w_1^2+w_1^3+w_1^4+(w_1^5+w_2^2w_1+w_3w_1^2)+(w_1^6+w_2^2w_1^2+w_3w_1^3)+\ldots
%	h[1] - h[1, 1] + h[1, 1, 1] - h[1, 1, 1, 1] + h[1, 1, 1, 1, 1] - h[1, 1, 1, 1, 1, 1] - h[2, 2, 1] + 3*h[2, 2, 1, 1] + h[3, 1, 1] - 3*h[3, 1, 1, 1]
	\\
		\ssu_{\Si^2(0)}=&\,(w_2^2+w_3w_1)+0+(w_2^2w_1^2+w_3w_1^3+w_3w_2w_1+w_3^2+w_4w_1^2+w_5w_1)+\ldots
%		h[2, 2] - 2*h[2, 2, 1] + 3*h[2, 2, 1, 1] - 2*h[2, 2, 2] - h[3, 1] + 2*h[3, 1, 1] - 3*h[3, 1, 1, 1] + 3*h[3, 2, 1] - 3*h[3, 3] - h[4, 1, 1] + 4*h[4, 2] - h[5, 1]
%		
\end{split}	
\end{equation}
%\begin{equation}#TANGENT VERSION
%	\begin{split}
%		\ssu_*(\Si^1(0))=&\,t_{1} + t_{1}^{2} + t_{1}^{3} + t_{1}^{4} + (t_{1}^{5} + t_{1}^{2} t_{3} + t_{1} t_{2}^{2}) + (t_{1}^{6} + t_{1}^{3} t_{3} + t_{1}^{2} t_{2}^{2})+\ldots\\
%		\ssu_*(\Si^2(0))=&\,(t_{1} t_{3} + t_{2}^{2}) + (t_{1}^{3} t_{3} + t_{1}^{2} t_{2}^{2} + t_{1}^{2} t_{4} + t_{1} t_{2} t_{3} + t_{1} t_{5} + t_{3}^{2})+\ldots
%	\end{split}	
%\end{equation}
and after rewriting in terms of tangent classes $t_i$ and multiplying with $(1+t_1+t_2+\ldots)$, we obtain the Stiefel-Whitney series defined in \eqref{eq:swhateta}:
\begin{equation}\label{eq:su_degeneracy}
	\begin{split}
		\hat{\su}_{\Si^1(0)}=&\,t_{1} + 0 + t_{1} t_{2} + (t_{1}^{2} t_{2} + t_{1} t_{3})+
		(t_{1}^{3} t_{2} + t_{1} t_{2}^{2} + t_{1} t_{4})+ (t_{1}^{4} t_{2} + t_{1}^{3} t_{3} + t_{1}^{2} t_{4} + t_{1} t_{5})+\ldots\\
		\hat{\su}_{\Si^2(0)}=&\,(t_{1} t_{3} + t_{2}^{2})+(t_{1}^{2} t_{3} + t_{1} t_{2}^{2})+(t_{1}^{3} t_{3} + t_{1}^{2} t_{2}^{2} + t_{1}^{2} t_{4} + t_{1} t_{5} + t_{2}^{3} + t_{3}^{2})+\ldots
	\end{split}	
\end{equation}
After rewriting the $\hat{\su}$-series into characteristic numbers $t_i\mapsto \tau_i$ as described in Section \ref{sec:eulerchar}, we can make further simplifications using the relations between the characteristic numbers. Let us list some of the relations between characteristic numbers (these can be verified by computing the characteristic numbers of the generators of unoriented cobordism or by methods of \cite{Dold1956}):
\[	t_I=\,0, \qquad |I|=1,3,5,7\text{ except for } t_I=t_3t_2\text{ and } t_5t_2=t_4t_2t_1=t_3t_2t_1^2\]
\[
t_2=\,t_1^2,\qquad t_3t_1=\,0,\qquad t_2t_1^2=\,t_1^4\]
\[a:=\,t_6=t_5t_1=t_4t_1^2=t_2^2t_1^2,\qquad b:=\,t_2^3=t_3^2=t_3t_2t_1,\qquad c:=\,t_1^4t_2=t_1^3t_3=t_4t_2=t_1^6,\]
%\[t_I=\,0,\qquad |I|=7,\text{ except for }.\]
%\begin{align}
%	t_I=\,&0, \qquad |I|=1,3,5\text{ except for } t_I=t_3t_2\\
%	t_2=\,&t_1^2,\\
%	t_3t_1=\,&0,\\
%	t_2t_1^2=\,&t_1^4\\
%	a:=\,&t_6=t_5t_1=t_4t_1^2=t_2^2t_1^2,\\
%	b:=\,&t_2^3=t_3^2=t_3t_2t_1,\\
%	c:=\,&t_1^4t_2=t_1^3t_3=t_4t_2=t_1^6\\
%	t_I=\,&0,\qquad |I|=7,\text{ except for } t_5t_2=t_4t_2t_1=t_3t_2t_1^2,
%\end{align}
Note that these are linear relations in the vector space of monomials, i.e.\ $t_3t_1=0$ does not imply $t_3t_2t_1=0$.
Using these relations, \eqref{eq:su_degeneracy} yields the following characteristic series:
\begin{equation}\label{eq:chi_degeneracy}
	\begin{split}	
		\chi_{\Si^1(0)}^*=\,&0+0+0+\tau_1^4+0+0+\ldots\\
		\chi_{\Si^2(0)}^*=\,&\tau_2^2+0 +0+\ldots\\
	\end{split}
\end{equation}

%\begin{equation}
%	\begin{split}
	%		\su_*(\Si^1(0))=&\,t_{1} +
	%		t_{1} t_{2} +
	%		(t_{1}^{2} t_{2} + t_{1} t_{3}) +
	%		(t_{1}^{3} t_{2} + t_{1} t_{2}^{2} + t_{1} t_{4}) +\\
	%		&+(t_{1}^{4} t_{2} + t_{1}^{3} t_{3} + t_{1}^{2} t_{4}  +t_{1} t_{5})+\ldots
	%	\end{split}
%\end{equation}
%\begin{equation}
%	\begin{split}
	%		\su_*(\Si^2(0))=&\,
	%		(t_{1} t_{3} + t_{2}^{2})
	%		+
	%		(t_{1}^{2} t_{3} + t_{1} t_{2}^{2})
	%		+\\
	%		&(t_{1}^{3} t_{3} + t_{1}^{2} t_{2}^{2} + t_{1}^{2} t_{4} + t_{1} t_{5} + t_{2}^{3} + t_{3}^{2} )+\ldots
	%	\end{split}
%\end{equation}

Passing to contact singularity loci, we can use Rim\'anyi's results \cite{tpp} together with Theorem \ref{thm:ssw} (5). We can do the same transformations as above for (open) singularity loci:
\begin{equation}\label{eq:ssu_singularity}
	\begin{split}
		\ssu_{A_2(0)}=&\,(w_2+w_1^2)+(w_3+w_1^3)+(w_4+w_2^2)+(w_5+w_2^2w_1)\\
		&+w_6+w_4w_2+w_2^3+w_3w_2w_1+w_4w_1^2+w_3w_1^3+w_2w_1^4+w_1^6+\ldots
		\\
		\ssu_{A_3(0)}=&\,(w_2w_1+w_1^3)+0+(w_4w_1+w_2^2w_1)+(w_3^2+w_4w_2+w_2^3+w_4w_1^2)+\ldots\\
		\ssu_{A_4(0)}=&\,(w_1^4 + w_1w_3) + (w_1^5 + w_1^2w_3) + (w_1^6 + w_1w_2w_3 + w_1w_5 + w_2^3 + w_2w_4 + w_3^2)\ldots\\
		\ssu_{\operatorname{I}_{22}(0)}=&\, (w_1w_3 + w_2^2)+0+(w_1^2w_4 + w_1w_5 + w_2^3 + w_2w_4 )+ \ldots\\
		\ssu_{A_5(0)}=&\,	(w_1^5 + w_1^2w_3) +
		(w_1^4w_2 + w_1^3w_3 + w_1^2w_2^2 + w_1^2w_4) +\ldots
	\end{split}
\end{equation}
%\begin{equation}\label{eq:ssu_singularity}
%	\begin{split}
%		\ssu_*(A_2(0))=&\,t_2+t_3+(t_2t_1^2+t_4)+(t_3t_1^2+t_5)+\\
%		&(t_2t_1^4+t_2^3+t_3t_2t_1+t_3^2+t_4t_1^2+t_4t_2+t_6)+\ldots\\
%		%\end{equation}
%		%\begin{equation}\label{eq:A3ssu}
%		\ssu_*(A_3(0))=&\,t_2t_1+0+(t_2t_1^3+t_4t_1)+(t_3^2+t_2t_4)+\ldots\\
%		%\end{equation}
%		%\begin{equation}\label{eq:A4ssu}
%		\ssu_*(A_4(0))=&\,t_3t_1+t_3t_1^2+(t_3t_2t_1+t_3^2+t_4t_1^2+t_4t_2+t_5t_1)+\ldots\\
%		\ssu_*(\operatorname{I}_{22}(0))=&\, (t_{1} t_{3} + t_{2}^{2}) + 0 (t_{1}^{3} t_{3} + t_{1}^{2} t_{2}^{2} + t_{1} t_{5} + t_{2} t_{4})+\ldots\\
%		\ssu_*(A_5(0))=&\,	t_{1}^{2} t_{3} + (t_{1}^{3} t_{3} + t_{1}^{2} t_{4})+\ldots
%	\end{split}
%\end{equation}
and after rewriting in terms of tangent classes and multiplying with $(1+t_1+t_2+\ldots)$ we obtain the Stiefel-Whitney series defined in \eqref{eq:swhateta}:
\begin{equation}\label{eq:su_singularity}
	\begin{split}
		\hat{\su}_{A_2(0)}=&\,t_2+(t_3+t_2t_1)+(t_4+t_1^2t_2+t_3t_1+t_2^2)+(t_5+t_4t_1+t_3t_1^2+t_2t_1^3)\\
		&(t_2t_1^4+t_2^2t_1^2+t_2^3 +t_3t_1^3+t_3t_2t_1+t_4t_1^2+t_4t_2+t_5t_1+t_6)+\ldots\\
		%	\end{split}
	%\end{equation}	
	%\begin{equation}\label{eq:A3su}
	%	\begin{split}
		\hat{\su}_{A_3(0)}=&\,t_2t_1+t_2t_1^2+(t_2t_1^3+t_2^2t_1+t_4t_1)+(t_2t_1^4+t_3t_2t_1+t_3^2+t_4t_1^2+t_4t_2)+\ldots\\
		%\end{equation}
		%\begin{equation}
		%\label{eq:A4su}
		\hat{\su}_{A_4(0)}=&\,t_3t_1+0+(t_3^2+t_4t_1^2+t_4t_2+t_5t_1+t_1^3t_3)+\ldots	\\
		%\end{equation}
		%\begin{equation}
		%	\label{eq:I22su}
		\hat{\su}_{\operatorname{I}_{22}(0)}=&\,	(t_{1} t_{3} + t_{2}^{2})+
		(t_{1}^{2} t_{3} + t_{1} t_{2}^{2})+
		(t_{1}^{3} t_{3} + t_{1}^{2} t_{2}^{2} + t_{1} t_{2} t_{3} + t_{1} t_{5} + t_{2}^{3} + t_{2} t_{4})+\ldots\\
		\hat{\su}_{A_5(0)}=&\,t_1^2t_3+t_1^2t_4 + \ldots
	\end{split}
\end{equation}
%\begin{equation}
%	\su_*(A_2(-1))=t_3t_1+t_3t_1^2+t_4t_1
%\end{equation}
Using the relations between characteristic numbers, we can write the characteristic series as
\begin{equation}\label{eq:chi_singularity}
	\begin{split}	
		\chi_{A_2(0)}=\,&\tau_2+0+(\tau_4+\tau_1^2\tau_2+\tau_2^2)+0+(\tau_2\tau_4)+\ldots\\
		\chi_{A_3(0)}=\,&0+\tau_1^4+0+\tau_6+\ldots\\
		\chi_{A_4(0)}=&\,0+0+\tau_3^2+\ldots	\\
		\chi_{\operatorname{I}_{22}(0)}=&\,\tau_2^2+0+0\ldots	\\
		\chi_{A_5(0)}=&\,0+\tau_6+\ldots	\\
	\end{split}
\end{equation}
Each of these coefficients corresponds to the Euler characteristic of the singularity locus of a generic map of an $n$-dimensional manifold to $\R^n$. For instance, $\chi^6_{A_4(0)}=\tau_3^2$ shows that the Euler characteristic of the $A_4$-points of a generic map $\RP^6\to\R^6$ is odd.

%\begin{align}
%	\chi_{\Si^1(0)}^*=\,&0+0+0+t_1^4+0+0+\ldots\\
%	\chi_{\Si^2(0)}^*=\,&t_2^2+0 +0+\ldots\\
%	\label{eq:A2chi}
%	\chi_{A_2(0)}^*=\,&0+0+(t_4+t_1^2t_2+t_2^2)+0+(t_2t_4)+\ldots\\
%	\chi_{A_3(0)}^*=\,&0+t_1^4+0+t_6+\ldots\\
%	\chi_{A_2(-1)}^*=\,&0+0+\ldots
%\end{align}

\end{document}